\documentclass[11pt,a4paper,oneside]{article}

\usepackage[english]{babel}
\usepackage{amsmath,amssymb,amsfonts,amsthm,mathrsfs}
\usepackage{yfonts}
\usepackage{graphicx,color}
\usepackage{epsfig}
\usepackage{multirow}
\usepackage{float}
\usepackage{hyperref}
\usepackage{enumitem}

\usepackage{pb-diagram}
\usepackage[all]{xy}
\usepackage[latin1]{inputenc}
\usepackage{amsmath,amsthm}
\usepackage[english]{babel}
\usepackage{latexsym}
\usepackage{amsfonts}
\usepackage{amssymb}
\usepackage{euscript}
\usepackage{color}
\usepackage{longtable}
\usepackage{tabularx}
\usepackage{pdfsync}
\usepackage{pstricks,pst-plot,pst-node,epsfig}
\usepackage{hyperref}
\usepackage{pgfplots}
\usepackage{multicol}

\usepackage{pb-diagram}

\def\a{\alpha}
\def\b{\beta}
\def\t{\theta}
\def\d{\delta}

\def\r{\varrho}
\def\g{\gamma}
\def\s{\sigma}
\def\t{\theta}
\def\l{\lambda}
\def\p{\partial}
\def\O{\Omega}
\def\e{\varepsilon}
\def\v{\varphi}
\def\G{\Gamma}
\def\k{\kappa}
\def\o{\omega}

\catcode`\@=11
\def\newremark#1{\@ifnextchar[{\@orem{#1}}{\@nrem{#1}}}
\def\@nrem#1#2{%
	\@ifnextchar[{\@xnrem{#1}{#2}}{\@ynrem{#1}{#2}}}
\def\@xnrem#1#2[#3]{\expandafter\@ifdefinable\csname #1\endcsname
	{\@definecounter{#1}\@addtoreset{#1}{#3}%
		\expandafter\xdef\csname the#1\endcsname{\expandafter\noexpand
			\csname the#3\endcsname \@remcountersep \@remcounter{#1}}%
		\global\@namedef{#1}{\@rem{#1}{#2}}\global\@namedef{end#1}{\@endremark}}}
\def\@ynrem#1#2{\expandafter\@ifdefinable\csname #1\endcsname
	{\@definecounter{#1}%
		\expandafter\xdef\csname the#1\endcsname{\@remcounter{#1}}%
		\global\@namedef{#1}{\@rem{#1}{#2}}\global\@namedef{end#1}{\@endremark}}}
\def\@orem#1[#2]#3{\expandafter\@ifdefinable\csname #1\endcsname
	{\global\@namedef{the#1}{\@nameuse{the#2}}%
		\global\@namedef{#1}{\@rem{#2}{#3}}%
		\global\@namedef{end#1}{\@endremark}}}
\def\@rem#1#2{\refstepcounter
	{#1}\@ifnextchar[{\@yrem{#1}{#2}}{\@xrem{#1}{#2}}}
\def\@xrem#1#2{\@beginremark{#2}{\csname the#1\endcsname}\ignorespaces}
\def\@yrem#1#2[#3]{\@opargbeginremark{#2}{\csname
		the#1\endcsname}{#3}\ignorespaces}
\def\@remcounter#1{\noexpand\arabic{#1}}
\def\@remcountersep{.}
\def\@beginremark#1#2{\rm \trivlist \item[\hskip \labelsep{\bf #1\ #2}]}
\def\@opargbeginremark#1#2#3{\rm \trivlist
	\item[\hskip \labelsep{\bf #1\ #2\ (#3)}]}
\def\@endremark{\endtrivlist}
\catcode`\@=12
\newcommand{\biindice}[3]%
{
	
	\begin{array}[t]{c}
		#1\\
		{\scriptstyle #2}\\
		{\scriptstyle #3}
	\end{array}
	
}
\def\a{\alpha}
\def\b{\beta}
\def\t{\theta}
\def\d{\delta}

\def\r{\varrho}
\def\g{\gamma}
\def\s{\sigma}
\def\t{\theta}
\def\l{\lambda}
\def\p{\partial}
\def\O{\Omega}
\def\e{\varepsilon}
\def\v{\varphi}
\def\G{\Gamma}
\def\k{\kappa}
\def\o{\omega}
\def\mc{\mathcal}
\def\mf{\mathfrak}
\def\ua{\uparrow}
\def\da{\downarrow}

\newcommand{\R}{\mathbb R}

\newcommand{\overbar}[1]{\mkern 1.5mu\overline{\mkern-1.5mu#1\mkern-1.5mu}\mkern 1.5mu}

\def\supp{\mbox{\rm supp\,}}

\numberwithin{equation}{section}

\theoremstyle{definition}

\theoremstyle{plain}
\newtheorem{theorem}{Theorem}[section]

\newtheorem{lemma}{Lemma}[section]
\newtheorem{corollary}{Corollary}[section]
\newtheorem{remark}{Remark}[section]



\title{\vskip-2.5cm
{A robust multiplicity result in a generalized diffusive predator-prey model}
\thanks{
This paper has been written under the auspices of the Ministry of Science and Innnovation of Spain under Reserach Grant PID2021-123343NB-I00, and  the Institute of Interdisciplinary Mathematics of
Complutense University of Madrid. The second author, ORCID: 0000-0003-1184-6231,  has been also supported by contract CT42/18-CT43/18 of Complutense University of Madrid} }
	
\author{
	\sc Juli\'an L\'opez-G\' omez
	\\
	\small Universidad Complutense de Madrid
	\\
	\small Instituto de Matem\'{a}tica Interdisciplinar (IMI)
	\\
	\small Departamento de An\'alisis Matem\'atico y Matem\'atica
	Aplicada
	\\
	\small  Plaza de las Ciencias 3, 28040   Madrid, Spain
	\\
	\small E-mail: {\tt   julian@mat.ucm.es}
	\medskip
	\\
	\sc Eduardo Mu\~{n}oz-Hern\'andez
	\\
	\small Universidad Complutense de Madrid
	\\
	\small Instituto de Matem\'{a}tica Interdisciplinar (IMI)
	\\
	\small Departamento de An\'alisis Matem\'atico y Matem\'atica
	Aplicada
	\\
	\small  Plaza de las Ciencias 3, 28040   Madrid, Spain
	\\
	\small E-mail: {\tt eduardmu@ucm.es }
	\bigskip
}
\begin{document}
	\maketitle

\begin{abstract}
This paper analyzes the generalized spatially heterogeneous diffusive predator-prey model introduced by the authors in  \cite{LGMH20}, whose interaction terms depend on a saturation coefficient $m(x)\gneq0$. As the amplitude of the saturation term, measured by $\|m\|_\infty$,  blows up to infinity, the existence of, at least, two coexistence states, is established in the region of the parameters where the semitrivial positive solution is linearly stable, regardless the sizes and the shapes of the remaining function coefficients in the setting of the model. In some further special cases, an $S$-shaped  component of coexistence states can be constructed, which causes the existence of, at least, three coexistence states, though
this multiplicity occurs within the parameter regions where the semitrivial positive solution is linearly unstable. Therefore, these multiplicity results inherit a rather different nature.
\end{abstract}

\section{Introduction}
	
\noindent This paper studies the existence and multiplicity of coexistence states for the generalized spatially heterogeneous predator-prey model
\begin{equation}
\label{1.1}
\left\{
\begin{array}{lll}
\mf{L}_1 u=\lambda u - a(x)u^2 - b(x)\dfrac{uv}{1+ \g m(x)u}&\quad \hbox{in}\;\;\Omega, \\[10pt]
\mf{L}_2 v=\mu v -d(x)v^2+ c(x)\dfrac{uv}{1+ \g m(x)u} &\quad \hbox{in}\;\;\Omega,\\[10pt]
\mf{B}_1 u=\mf{B}_2 v=0 &\quad\hbox{on}\;\;\partial\Omega,
\end{array}
\right.
\end{equation}
where $\Omega$ is a bounded domain of $\R^N$ with boundary, $\partial\Omega$,  of class $\mc{C}^2$,  and $\mf{L}_\kappa$, $\kappa=1,2$, are second order uniformly elliptic operators in $\O$ of the form
\begin{equation}
\label{1.2}
		\mf{L}_\kappa:=-\mathrm{div\,}(A_\kappa \nabla ) +\langle b_\kappa,\nabla\rangle +c_\kappa, \qquad \kappa=1,2,
\end{equation}
where, for every $\kappa=1,2$,
\begin{equation*}
   A_\kappa=\left(a_{ij}^\kappa\right)_{1\leq i,j \leq N}\in
   \mathscr{M}_N^\mathrm{sym}(W^{1,\infty}(\O)), \quad
   b_\kappa=(b_1^\kappa,...,b_N^\kappa)\in (L^\infty(\O))^N,\quad c_\kappa \in L^\infty(\O).
\end{equation*}
For a given Banach space $X$, we are denoting by $\mathscr{M}_N^\mathrm{sym}(X)$ the space of the symmetric
square matrices of order $N$ with entries in $X$, and $W^{1,\infty}(\O)$ stands for the Sobolev space of
all bounded and measurable functions in $\O$ with weak derivatives in $L^\infty(\O)$. In
\eqref{1.1}, for every $\kappa=1, 2$, $\mf{B}_\kappa$ is a general boundary operator of mixed
type such that, for every $\psi \in \mc{C}(\bar\O)\cap\mc{C}^1(\O\cup\G_1^\kappa)$,
\begin{equation}
		\label{1.3}
		\mf{B}_\kappa \psi =\left\{ \begin{array}{ll}
			\psi &\quad\hbox{on}\;\;\Gamma_0^\kappa,\\[5pt]
			\partial_{\nu_\kappa} \psi+ \b_\kappa(x)\psi&\quad\hbox{on}\;\;\Gamma_1^\kappa,
		\end{array}
		\right.
	\end{equation}
where $\Gamma_0^\kappa$ and $\Gamma_1^\kappa$ are two closed and open disjoint subsets of
$\partial\Omega$ such that $\Gamma_0^\kappa\cup\Gamma_1^\kappa=\partial\Omega$, and
$\nu_\kappa= A_\kappa n$ is the co-normal vector field, i.e., $n$ is the outward unit normal vector field
of $\O$. In \eqref{1.3}, $\b_\kappa \in \mc{C}(\G_1^\kappa)$ is not required to have any special sign.
As for the coefficient functions $a(x)$, $b(x)$, $c(x)$, $d(x)$ and $m(x)$ in the setting of \eqref{1.1},
we assume that they are functions in $\mc{C}(\bar\O;\R)$ such that $b\neq0$, $c\neq 0$, and
\begin{equation}
\label{1.4}
a(x)>0,\;\; d(x)>0, \;\; b(x)\geq 0,\;\; c(x)\geq 0,\;\; m(x)\geq 0 \quad\hbox{for all}\;\, x\in\bar\O.
\end{equation}
In other words, $a\gg 0$, $d\gg 0$, $b\gneq 0$, $c\gneq 0$ and
$m\geq 0$. Finally, in \eqref{1.1}, $\l$, $\mu$ and $\g>0$ are regarded as real parameters.
\par
Except for the incorporation of the new parameter $\g>0$,  this  model, in its greatest generality, was introduced by the authors in \cite{LGMH20} to establish an homotopy between the classical diffusive Lotka--Volterra predator-prey system, when $m=0$, and the diffusive Holling--Tanner model introduced by Casal, Eilbeck and L\'{o}pez-G\'{o}mez  \cite{CELG}, where $m$ is a positive constant. The case when $m$ is constant has been also analyzed by Du and Lou in \cite{DL-1997}, \cite{DL-1998} and \cite{DL-2001}, under Dirichlet or Neumann boundary conditions, and Du and Shi \cite{DS-2006}  assuming the existence of a protection zone for the prey.
Some pioneering non-spatial models of this type were studied by Freedman \cite{HFR},  May \cite{May} and Hsu \cite{Hsu}, among others.
\par
In Population Dynamics, \eqref{1.1} represents the interaction in a common habitat, $\O$,
between a prey, with population density $u$, and a predator, with population density $v$.  According to \eqref{1.1}, in the absence of the other, each species has a  logistic growth determined by the relative sizes of $\l$ and $\mu$ with respect to the thresholds   $\s_0[\mf{L}_1-c_1,\mf{B}_1,\O]$ and
$\s_0[\mf{L}_2-c_2,\mf{B}_2,\O]$, respectively. Throughout this paper, for any given second order
elliptic operator $\mf{L}$ in $\O$ and any boundary operator $\mf{B}$ on $\p\O$, we denote by
$\s_0[\mf{L},\mf{B},\O]$ the principal eigenvalue of $(\mf{L},\mf{B},\O)$ as discussed in
\cite{LG01}.  In  \eqref{1.1}, the term $\g m(x)$ measures the saturation effects in $\O$ of the predator in the presence of a high population of preys. More precisely, for every $x\in\O$, $\g m(x)$ measures the predator saturation level at the location $x\in\O$ if $m(x)>0$, while the saturation effects at $x$ do not play any role if $m(x)=0$. By normalizing $m(x)$ so that $\max_{x\in\bar \O}m(x)=1$, $\g$ becomes the maximal intensity of the saturations effects. So, throughout this paper we will assume that
\begin{equation}
\label{1.5}
   \|m\|_\infty\equiv \max_{x\in\bar\O}\,m(x)=1.
\end{equation}
Furthermore, we assume that $\O_0:= \mathrm{int\,}m^{-1}(0)$ is a nice open subset of class  $\mc{C}^2$ of $\O$ with finitely many connected components and $\bar \O_0=m^{-1}(0)\subset \O$.
Thus, \eqref{1.1} combines in the same habitat, $\O$,  functional responses of Lotka--Volterra type in the components of $m^{-1}(0)$ together with Holling--Tanner responses in $m^{-1}(\R_+)$, where $\R_+:=(0,+\infty)$.
As noticed in Sections 3 and 5 of \cite{LGMH20}, the existence of both functional responses can lead to global effects in the dynamics of the species, regardless the sizes of the patches where $m=0$ or $m>0$. Moreover, the size of the regions where $m(x)$ or $b(x)$ degenerate can also affect the global dynamics. Indeed, as shown in Section \ref{sec4}, the greater is the support of $m(x)$, or $b^{-1}(0)$, the smaller can be $\l$ so that \eqref{1.1} can still admit a coexistence state.
\par
Essentially, this paper is a continuation of \cite{LGMH20}, where the existence and the uniqueness of coexistence states was established for the generalized problem \eqref{1.1}, by fixing $\mu\in\R$ and regarding  $\l\in\mathbb{R}$ as a bifurcation parameter. According to Theorem 7.1 of
\cite{LGMH20}, we already know that the one-dimensional counterpart of \eqref{1.1} has a unique
coexistence state for sufficiently small $\g>0$. The main goal of this paper is to study the dynamics of \eqref{1.1} as $\g\ua +\infty$. Thus, it is rather natural to perform the change of variables
\begin{equation}
\label{1.6}
w:=\g \,u,\qquad \varepsilon=\frac{1}{\g}.
\end{equation}
In these variables, \eqref{1.1} can be expressed, equivalently, as
\begin{equation}
\label{1.7}
\left\{
\begin{array}{lll}
\mf{L}_1 w=\lambda w - \varepsilon a(x)w^2 -b(x)\dfrac{wv}{1+m(x)w} &\quad \hbox{in}\;\;\Omega,\\[10pt]
\mf{L}_2 v=\mu v -d(x)v^2+ \varepsilon c(x)\dfrac{wv}{1+m(x)w}&\quad \hbox{in}\;\;\Omega,\\[10pt]
\mf{B}_1 w=\mf{B}_2 v=0 &\quad\hbox{on}\;\;\partial\Omega.
\end{array}
\right.
\end{equation}
According to \eqref{1.6}, analyzing the dynamics of \eqref{1.1} for sufficiently large $\g$
is equivalent to analyze \eqref{1.7} for sufficiently small $\e>0$. Thus, it is rather natural to focus attention into \eqref{1.7} as a system perturbing from
\begin{equation}
\label{1.8}
\left\{
\begin{array}{lll}
\mf{L}_1 w=\lambda w -b(x)\dfrac{wv}{1+m(x)w} &\quad \hbox{in}\;\;\Omega,\\[10pt]
\mf{L}_2 v=\mu v -d(x)v^2&\quad \hbox{in}\;\;\Omega,\\[10pt]
\mf{B}_1 w=\mf{B}_2 v=0 &\quad\hbox{on}\;\;\partial\Omega.
\end{array}
\right.
\end{equation}
This problem has the tremendous advantage that it is  uncoupled.
\par
Our main results establish, for every $\e\geq 0$,  the existence of a component $\mathscr{C}_\e^+$ of the set of coexistence states of \eqref{1.7}, or $\eqref{1.8}$,  and ascertain their global structures according to weather $\e>0$, or $\e=0$. Precisely, when $\e=0$, Theorems \ref{th4.1} and \ref{th4.2} show that $\mathscr{C}_0^+$ behaves much like sketched in Figure \ref{Fig3}, where the constants $\Phi(\mu)$ and $\v_0(\mu)$ are defined in \eqref{3.16} and \eqref{3.19}, respectively. Later, Theorem \ref{th5.1} shows that, as $\e>0$ perturbs from $\e=0$, the component $\mathscr{C}_0^+$ perturbs into $\mathscr{C}_\e^+$ and that,
since  the coexistence states of \eqref{1.7} have uniform a priori bounds on compact subintervals of the
parameter $\l$, for any given $\eta>0$, there exists $\e_0=\e_0(\eta)>0$ such that $\mathscr{C}_\e^+$ has, at least, two coexistence states for every $\l\in [\v_0(\mu)-\eta,\Phi(\mu)-\eta]$ if $\e \in (0,\e_0]$, as illustrated by Figure \ref{Fig5}. This multiplicity result is new even for the simplest prototype model introduced by Casal et al. \cite{CELG}.
\par
Although in the classical setting of Casal et al. \cite{CELG}, Du and Lou \cite{DL-1998} proved the existence of the $S$-shaped diagrams computed in \cite{CELG} for sufficiently large $\g>0$ and $c>0$, with $\mu>\s_{0,2}\equiv \s_0[\mf{L}_2,\mf{B}_2,\O]$ sufficiently close to $\s_{0,2}$, the reader should be aware that, in this paper, $c(x)$ can degenerate and take arbitrary values, and that $\mu>\s_{0,2}$ is arbitrary. Actually, the multiplicity result of this paper has a different nature than the inherent to the $S$-shaped diagrams discovered  in \cite{CELG}. In $S$-shaped bifurcation diagrams, the problem has, at least, two coexistence states if $\l\in [\Phi(\mu)-\eta,\Phi(\mu)]$, while it has, at least, three, if $\l\in (\Phi(\mu),\Phi(\mu)+\eta]$, for sufficiently small $\eta>0$, as illustrated in the second picture of Figure \ref{Fig8}. In strong contrast, the main result of this paper establishes that, for sufficiently large $\g>0$, \eqref{1.1} has, at least,  two coexistence states in any compact subinterval of $(\v_0(\mu),\Phi(\mu))$, regardless the size and shape of the function coefficient $c(x)$ and how large is $\mu$. Rather surprisingly,
this occurs regardless the size of the support of the saturation term, measured by $m(x)$, which might be 
arbitrarily small, as is an atom in a Galaxy. A similar phenomenon, though in a very different problem,
was observed by L\'{o}pez-G\'{o}mez and Rabinowitz \cite{LGRab}. 
\par
We end this paper by analyzing a simple prototype model with constant coefficients and
non-flux boundary conditions, where the constant steady-states are given by a simple algebraic system.
Among other things, we will establish the existence of $S$-shaped curves of coexistence states
when $bc>ad$ and $\e$ is sufficiently large. This example shows that our multiplicity theorem, for sufficiently small $\e>0$,  has nothing to do with the formation of $S$-shaped components of coexistence states.
\par
The plan of this paper is the following. Section \ref{sec2} introduces some notations and abstract results that are used throughout the paper. Section \ref{sec3} studies the stability of the semitrivial curve $(0,\t_{[\mf{L}_2,\mf{B}_2,\O]})$, where $\t_{[\mf{L}_2,\mf{B}_2,\O]}$ stands for the unique positive solution of
\begin{equation*}
\left\{
\begin{array}{ll}
\mf{L}_2=\mu v-d v^2
&\quad \hbox{in}\;\O,\\[1ex]
\mf{B}_2v=0&\quad \hbox{on}\;\partial\O,
\end{array}
\right.
\end{equation*}
which exists if, and only if, $\mu>\s_{0,2}$, and analyzes the local bifurcation to coexistence states of \eqref{1.7} from it, with special emphasis on the uniform dependence of these local bifurcations on the parameter $\e\geq 0$, which is a subtle issue. Section \ref{sec4} studies the uncoupled system \eqref{1.8}, establishing the global structure  of the component $\mathscr{C}_0^+$ near $\v_0(\mu)$, its bifurcation point from infinity, and $\Phi(\mu)$, its bifurcation point from the semitrivial positive solution
$(0,\t_{[\mf{L}_2,\mf{B}_2,\O]})$. Then, the analysis carried out in Sections \ref{sec3} and \ref{sec4} combined  with some sophisticated topological and global
continuation arguments, will drive us to the proof of
Theorem \ref{th5.1} of Section \ref{sec5}, which is our main multiplicity result. Finally, in Section \ref{sec6} we analyze a very simple example with $S$-shaped components of coexistence states. A previous analysis of this example is imperative for tackling the problem of the global existence of $S$-shaped bifurcation diagrams in its greatest generality, which will be pursued in a  forthcoming paper.

\section{Preliminaries}
\label{sec2}
This section collects some results scattered in a series of papers and monographs that are going to be used throughout this paper. As a direct consequence of the elliptic $L^p$-theory (see, e.g.,
Chapters 4 and 5 of \cite{LG13}), it becomes apparent that any non-negative weak solution of \eqref{1.7}, $(w,v)$, satisfies
\begin{equation*}
	w\in \mathscr{W}_1 \equiv  \bigcap_{p\geq N} W^{2,p}_{\mf{B}_1}(\O),
    \qquad v\in \mathscr{W}_2 \equiv  \bigcap_{p\geq N}W^{2,p}_{\mf{B}_2}(\O),
\end{equation*}
where, for every  $\kappa=1, 2$ and $p\geq N$,  $W^{2,p}_{\mf{B}_\kappa}(\O)$ stands for the Sobolev space of the functions $z\in W^{2,p}(\O)$ such that $\mf{B}_\kappa z=0$ on $\p\O$. Thus, $(w,v)$ is a strong solution of \eqref{1.7}. In particular, $w$ and $v$ are twice classically differentiable almost everywhere in $\O$ and they are classical solutions in the sense of \cite[Def. 4.1]{LG13}. By the Sobolev embeddings and the
Rellich--Kondrashov theorem, it is easily seen that $ \mathscr{W}_\k \hookrightarrow \mc{C}_{\mf{B}_\k}^1(\bar\O)$, $\k=1, 2$, with compact embeddings, where $\mc{C}_{\mf{B}_\k}^1(\bar\O)$
stands for the set of functions $z\in \mc{C}^1(\bar\O)$ such that $\mf{B}_\k z =0$ on $\p\O$
(see \cite[Ch. 4]{LG13} if necessary).

Throughout this paper, for every weight function  $V\in L^\infty(\O)$ and $\kappa=1, 2$, we denote by
$\s_0[\mf{L}_\kappa+V,\mf{B}_\kappa,\O]$ the principal eigenvalue of the linear eigenvalue problem
\begin{equation}
\label{ii.1}
	\left\{ \begin{array}{ll} \left(\mf{L}_\kappa+V\right)\v =\tau \v & \quad \hbox{in}\;\; \O, \\
    \mf{B}_\kappa \v =0 & \quad \hbox{on}\;\;\p\O, \end{array}\right.
\end{equation}
whose existence and uniqueness in our general setting was established by \cite[Th. 7.7]{LG13}. According to
Corollary 7.1 and Theorem 7.9 of \cite{LG13},  $\s_0[\mf{L}_\kappa+V,\mf{B}_\kappa,\O]$ is strictly dominant and algebraically simple. In particular, it is the lowest real eigenvalue. Moreover, by \cite[Th. 7.6]{LG13}, for every $\kappa=1, 2$,  the associated principal eigenfunction, unique up to a multiplicative positive constant, can be taken to be strongly positive in $\O$, $\v\gg_\k 0$, in the sense that
\begin{equation*}
	\v(x)>0 \;\;\hbox{for all}\;\; x\in\O\cup\G_1^\kappa \;\;\hbox{and}\;\; \frac{\p \v}{\p n}(x)<0 \;\;\hbox{for all}\;\; x\in \G_0^\kappa,
\end{equation*}
where $n$ stands for the outward unit vector field to $\O$ along $\p\O$. Subsequently, we
collect some important results that are going to be invoked throughout this paper. The first one, going back to Cano-Casanova and L\'{o}pez-G\'{o}mez \cite{CCLG} in its present generality, establishes the monotonicity of the principal eigenvalue with respect to the potential.

\begin{theorem}
\label{th2.1}
Let $V_1,V_2\in L^{\infty}(\Omega)$ be such that $V_1\lneq V_2$. Then, for every $\k=1, 2$,
\[
   \sigma_0\left[\mf{L}_\k+V_1,\mf{B}_\k,\O\right]<\sigma_0\left[\mf{L}_\k+V_2,\mf{B}_\k,\O\right].
\]
Thus, the map $V\mapsto \sigma_0\left[\mf{L}_k+V,\mf{B}_k,\O\right]$ is continuous in $L^\infty(\O)$  and increasing.
\end{theorem}

The next characterization theorem is \cite[Th. 7.10]{LG13}. It goes back to L\'{o}pez-G\'{o}mez and Molina-Meyer \cite{LGMM} for cooperative systems under Dirichlet boundary conditions, and to Amann and L\'{o}pez-G\'{o}mez \cite{ALG-98} in the present setting. The equivalence between (a) and (c) was established, simultaneously to \cite{LGMM},  for the single equation under Dirichlet boundary conditions   by Berestycki, Nirenberg and Varadhan \cite{BNV}. However, (b) is the most useful condition
from the point of the applications.
\par

\begin{theorem}
	\label{th2.2}
	For every $V\in L^\infty(\O)$ and $\k=1, 2$, the next conditions are equivalent:
	\begin{enumerate}
		\item[{\rm (a)}] $\sigma_0\left[\mf{L}_\k+V,\mf{B}_\k,\Omega\right]>0$.
		\item[{\rm (b)}] The tern $\left(\mf{L}_\k+V,\mf{B}_\k,\Omega\right)$ possesses a positive strict supersolution, $h \in \mathscr{W}_\k$, i.e., $h$ satisfies $h \gneq 0$ and 		
\begin{equation*}
			\left \{
			\begin{array}{ll}
				(\mf{L}_\k+V) h\geq 0&\quad\hbox{in}\;\;\Omega,\\
				\mf{B}_\k h\geq 0&\quad\hbox{on}\;\;\partial\Omega,
			\end{array}
			\right.
\end{equation*}
		with some of these inequalities strict.
\item[{\rm (c)}] The tern $(\mf{L}_\k+V,\mf{B}_\k,\Omega)$ satisfies the strong maximum principle, i.e.,  every function $z\in \mathscr{W}_\k$ such that
		\begin{equation*}
			\left \{
			\begin{array}{ll}
				(\mf{L}_\k+V) z\geq 0&\quad\hbox{in}\;\;\Omega,\\
				\mf{B}_\k z \geq 0&\quad\hbox{on}\;\;\partial\Omega, \end{array} \right.
		\end{equation*}
		with some of these inequalities strict, satisfies
\begin{equation*}
			z(x)>0 \;\;\hbox{for all}\;\; x\in\O\cup\G_1^\k \;\;\hbox{and}\;\; \frac{\p z}{\p n}(x)<0 \;\;\hbox{for all}\;\; x\in z^{-1}(0) \cap \G_0^\k.
		\end{equation*}
To shorten  notations, when this occurs, we will simply say that $z\gg_\k 0$.
\end{enumerate}
\end{theorem}
\par

The next result goes back to Fraile et al. \cite[Th. 3.5]{FKLGM} for $\b_\k\geq 0$. In the general case when $\b_\k$ changes sign one can either use the change of variable of Fern\'{a}ndez-Rinc\'{o}n and L\'{o}pez-G\'{o}mez \cite[Sect. 3]{FRLG} to reduce the problem to the setting of \cite{FKLGM}, or one might derive it directly from Theorem 1.1 of Daners and L\'{o}pez-G\'{o}mez \cite{DLG}. Subsequently, we say that
$z_1\ll_\k z_2$ if $z_2-z_1\gg_\k 0$.

\begin{theorem}
\label{th2.3}
Suppose $\r \in\R$ and $\xi \in C(\bar \Omega;\R)$ satisfies $\xi(x)>0$ for all $x\in\bar\O$. Then, for every $\k=1, 2$ and $V\in L^\infty(\O)$, the semilinear boundary value problem
\begin{equation}
\label{ii.2}
\left\{
\begin{array}{ll}
(\mf{L}_\k +V) z=\r z - \xi(x) z^2 &\quad\hbox{in}\;\;\Omega,\\[5pt]
\mf{B}_\k z=0    &\quad\hbox{on}\;\;\partial\Omega,
\end{array}
\right.
\end{equation}
admits a positive solution if, and only if, $\r > \r_\k\equiv \sigma_0\left[\mf{L}_\k+V,\mf{B}_\k,\Omega\right]$. Moreover, it is unique if it exists, and, denoting it by $z_{\r,\k}\equiv \t_{[\mf{L}_\k+V,\r,\xi]}$,
we have that $w_{\r,\k}\gg_\k 0$ and
\begin{enumerate}
\item[{\rm (a)}] the map  $\r\mapsto w_{\r,\k}$ is point-wise increasing provided
$\r>\r_\k$,
\item[{\rm (b)}] $z_{\r,\k}$ bifurcates from $z=0$
at $\r =\r_\k$,
\item[{\rm (c)}] as a consequence of Theorem  \ref{th2.2}, if $\bar{u}$ (resp. $\underline{u}$) is a positive strict supersolution (resp. subsolution) of \eqref{ii.2}, then $\underline{u}\ll_{\k} w_{\r,\k}$ (resp. $ w_{\r,\k} \ll_{\k}\bar{u}$) provided $\r >\r_\k$.
\end{enumerate}
\end{theorem}

More precisely, in this paper we denote by $\t_{[\mf{L}_\k +V,\r,\xi]}$
the maximal non-negative solution of \eqref{ii.2}. Then, due to Theorem \ref{th2.3},
\begin{equation*}
 \t_{[\mf{L}_\k +V,\r,\xi]}:= \left\{ \begin{array}{ll} 0 & \quad \hbox{if}\;\;
\r \leq \r_\k,\\ \gg_\k 0 & \quad \hbox{if}\;\;
\r > \r_\k.
\end{array}
\right.
\end{equation*}
Theorem \ref{th2.3} was generalized by Fraile et al. \cite{FKLGM} to cover the case when $\xi\gneq 0$
vanishes on some nice subdomain of $\Omega$, and by Daners and L\'opez-G\'omez \cite[Th. 1.1]{DLG}
to characterize the range of $\r$'s for which \eqref{ii.2} admits a positive solution under no requirements on the nature of $\xi^{-1}(0)$.

\begin{corollary}
	\label{co2.1}
	According to Theorem \ref{th2.3}, we can conclude that
\begin{enumerate}
\item[{\rm (a)}]
\eqref{1.1} has a semitrivial positive solution of the form $(u,0)$
	if, and only if, $\l > \s_{0,1}\equiv \s_0[\mf{L}_1,\mf{B}_1,\Omega]$, and, in such case,
$u=\t_{[\mf{L}_1,\l,a]}$.

\item[{\rm (b)}] Similarly, \eqref{1.1} has a semitrivial positive solution of the form $(0,v)$
	if, and only if, $\mu > \s_{0,2}\equiv \s_0[\mf{L}_2,\mf{B}_2,\Omega]$, and, in such case,  $v=\t_{[\mf{L}_2,\mu,d]}$.
\end{enumerate}
\end{corollary}

\section{Bifurcation of coexistence states from $(0,\t_{[\mf{L}_2,\mu,d]})$}
\label{sec3}

In this section we analyze the bifurcation of coexistence states from the semitrivial curve $(0,\t_{[\mf{L}_2,\mu,d]})$ in the problem \eqref{1.7}. We are particularly interested in ascertaining
the nature of the local bifurcation according to the value of the parameter $\e>0$.  The linearized stability of $(0,\t_{[\mf{L}_2,\mu,d]})$ is determined by the signs of the real parts of the eigenvalues of the problem
\begin{equation}
\label{3.1}
\left\{
\begin{array}{ll}
\begin{pmatrix}
\mf{L}_1+b\t_{[\mf{L}_2,\mu,d]}-\l& 0 \\[1ex]
-\e c\t_{[\mf{L}_2,\mu,d]} & \mf{L}_2+2d\t_{[\mf{L}_2,\mu,d]}-\mu
\end{pmatrix}
\begin{pmatrix} w \\[1ex] v \end{pmatrix}
= \tau \begin{pmatrix} w \\[1ex] v \end{pmatrix}
&\quad \hbox{in}\;\O,\\[17pt]
\mf{B}_1 w=\mf{B}_2v=0&\quad \hbox{on}\;\partial\O.
\end{array}
\right.
\end{equation}
The next result holds.

\begin{theorem}
\label{th3.1}
Setting $\Phi(\mu)\equiv \s_0\left[ \mf{L}_1+b \t_{[\mf{L}_2,\mu,d]},\mf{B}_1,\O\right]$ for all $\mu > \s_0[\mf{L}_2,\mf{B}_2,\O]$, the semitrivial solution $(0,\t_{[\mf{L}_2,\mu,d]})$ is linearly unstable if, and only if,
$\l>\Phi(\mu)$, whereas it is linearly stable if, and only if, $\l<\Phi(\mu)$. Thus,
$\l=\Phi(\mu)$ is the curve of change of stability of $(0,\t_{[\mf{L}_2,\mu,d]})$.
\end{theorem}
\begin{proof}
We first determine the eigenvalues with associated eigenvectors $(w,v)$ such that $w=0$ and $v\neq 0$.
By \eqref{3.1}, these eigenvalues satisfy
\begin{equation}
\label{3.2}
		\left\{
		\begin{array}{ll}
			(\mf{L}_2+2d\t_{[\mf{L}_2,\mu,d]}-\mu)v=\tau v
			&\quad \hbox{in}\;\O,\\[1ex]
			\mf{B}_2v=0&\quad \hbox{on}\;\partial\O.
		\end{array}
		\right.
\end{equation}
By Theorem \ref{th2.1}, the definition of $\t_{[\mf{L}_2,\mu,d]}$, and the uniqueness of the principal eigenvalue,
\begin{equation}
\label{3.3}
		\s_0\left[ \mf{L}_2+2d \t_{[\mf{L}_2,\mu,d]}-\mu,\mf{B}_2,\O\right]>\s_0\left[ \mf{L}_2+d \t_{[\mf{L}_2,\mu,d]}-\mu,\mf{B}_2,\O\right]=0.
\end{equation}
Thus, by the dominance of the principal eigenvalue (see \cite[Th. 7.8]{LG13}),
$$
  \mathrm{Re\,}\tau \geq \s_0\left[ \mf{L}_2+2d \t_{[\mf{L}_2,\mu,d]}-\mu,\mf{B}_2,\O\right]>0
$$
for any eigenvalue, $\tau$, of \eqref{3.2}.
\par
Now, we will ascertain the real parts of the eigenvalues of \eqref{3.1} with associated eigenfunctions, $(w,v)$, such that $w\neq 0$. By \eqref{3.1}, they should satisfy
\begin{equation}
\label{3.4}
		\left\{
		\begin{array}{ll}
			(\mf{L}_1+b\t_{[\mf{L}_2,\mu,d]}-\l)w=\tau w
			&\quad \hbox{in}\;\O,\\[1ex]
			\mf{B}_1 w=0&\quad \hbox{on}\;\partial\O.
		\end{array}
		\right.
\end{equation}
These eigenvalues consist of the sequence
\begin{equation}
		\label{eig}
		\tau_j:=\s_j\left[ \mf{L}_1+b \t_{[\mf{L}_2,\mu,d]},\mf{B}_1,\O\right]-\l\qquad\hbox{for}\;\,j\geq0,
\end{equation}
where $\{\s_j\left[ \mf{L}_1+b \t_{[\mf{L}_2,\mu,d]},\mf{B}_1,\O\right]\}_{j\geq 0}$ is the sequence of eigenvalues of \eqref{ii.1} with $\k=1$ and $V=b \t_{[\mf{L}_2,\mu,d]}$. As the principal eigenvalue is dominant, it is apparent that
$$
  \mathrm{Re\,}\tau_j\geq \tau_0=\Phi(\mu)-\l \quad \hbox{for all}\;\; j\geq 0.
$$
Assume $\l<\Phi(\mu)$. Then, $\mathrm{Re\,}\tau_j >0$ for all $j\geq 0$. Thus, any eigenvalue of
\eqref{3.1} has a positive real part, i.e., $(0,\t_{[\mf{L}_2,\mu,d]})$ is linearly stable.
\par
Assume now $\l >\Phi(\mu)$. Then, $\tau_0<0$. Let $w\neq 0$ be a principal eigenfunction associated to
$\tau_0$. Then, the second equation of \eqref{3.1} becomes
\begin{equation}
\label{3.6}
	\left\{ \begin{array}{ll} (\mf{L}_2+2d\t_{[\mf{L}_2,\mu,d]}-\mu-\tau_0)v=\e c\t_{[\mf{L}_2,\mu,d]}w, & \quad \hbox{in}\;\;\O,\\[1ex] \mf{B}_2 v =0 &\quad \hbox{in}\;\;\p\O. \end{array}\right.
\end{equation}
Since $-\tau_0>0$, it follows from Theorem \ref{th2.1} and \eqref{3.3} that
$$
  \s_0[\mf{L}_2+2d\t_{[\mf{L}_2,\mu,d]}-\mu-\tau_0,\mf{B}_2,\O]>
  \s_0[\mf{L}_2+2d\t_{[\mf{L}_2,\mu,d]}-\mu,\mf{B}_2,\O]>0.
$$
Thus,  thanks to Theorem \ref{th2.2},
\[
	v=\left(\mf{L}_2+2d\t_{[\mf{L}_2,\mu,d]}-\mu-\tau_0\right)^{-1}(\e c\t_{[\mf{L}_2,\mu,d]}w)
\]
provides us with the unique solution of \eqref{3.6}. Therefore, $(w,v)$ is an eigenfunction of
\eqref{3.1} associated to $\tau_0<0$ and hence,  $(0,\t_{[\mf{L}_2,\mu,d]})$  is linearly
unstable.
\end{proof}
\par
\begin{remark}
\label{re3.1}
\rm According to the theorems of Lyapunov on linearized stability, it becomes apparent that
$(0,\t_{[\mf{L}_2,\mu,d]})$ is  exponentially asymptotically stable if $\l<\Phi(\mu)$, while it is
unstable if $\l>\Phi(\mu)$ (see, e.g., Henry \cite[Sec. 5.1]{Hen}).
\end{remark}

Subsequently, we set
\begin{equation}
\label{3.7}
  \s_{0,\k}\equiv \s_0[\mf{L}_\k,\mf{B}_\k,\O],\qquad \k=1, 2,
\end{equation}
and pick any real number, $e$, such that $e> \max\{-\s_{0,1},-\s_{0,2}\}$. Then, for every $\k=1, 2$,
$$
  \s_0[\mf{L}_\k+e,\mf{B}_\k,\O]=\s_{0,\k}+e>0
$$
and hence, by Theorem \ref{th2.2}, $(\mf{L}_\k+e,\mf{B}_\k,\O)$ is an invertible operator with
strongly positive inverse. Obviously, the solutions of the problem \eqref{1.7} are given by the zeroes
of the operator
$$
\mf{F}:\mathbb{R}\times \mathbb{R}\times\mathbb{R} \times \mc{C}^1_{\mf{B}_1}(\bar\O)\times \mc{C}^1_{\mf{B}_2}(\bar\O) \to \mathscr{W}_1\times \mathscr{W}_2,
$$
defined, for every $\l,\mu, \e \in\R$, $w\in \mc{C}^1_{\mf{B}_1}(\bar\O)$ and $v \in\mc{C}^1_{\mf{B}_2}(\bar\O)$, by
\begin{equation}
\label{3.8}
\mf{F}(\l,\mu,\e,w,v):=\left( \begin{array}{c}
		w- (\mf{L}_1+e)^{-1}\left[ (\l+e)w- \e a w^2 - b \frac{wv}{1+mw} \right] \\[7pt]
		v - (\mf{L}_2+e)^{-1} \left[ (\mu+e)v-dv^2+\e c \frac{wv}{1+mw} \right] \end{array} \right).
\end{equation}
The operator $\mf{F}$ is a compact perturbation of the identity map in $\mc{C}^1_{\mf{B}_1}(\bar\O)\times
\mc{C}^1_{\mf{B}_2}(\bar\O)$. Moreover, it is Fr\'{e}ch\`{e}t differentiable and, since  $D_{(w,v)}\mf{F}$ is a linear compact perturbation of the identity map, $D_{(w,v)}\mf{F}$ is a Fredholm operator of index zero. Actually, $\mf{F}$ is real analytic in an open region containing the first quadrant $w\geq 0$, $v\geq 0$.
\par
The next result shows that the coexistence states bifurcate from the
semitrivial positive solution $(0,\t_{[\mf{L}_2,\mu,d]})$ along the curve $\l=\Phi(\mu)$. It is a
direct consequence of the theorem of bifurcation from simple eigenvalues of Crandall and Rabinowitz \cite{Rab-71b}. It provides us with the local structure of the set of bifurcating
coexistence states.

\begin{theorem}
\label{th3.2}
For every $\mu>\s_{0,2}$ and $\e\in\R$, there exist $\d=\d(\mu,\e)>0$ and an analytic map
$(\l,w,v):(-\d,\d)\to \R \times \mathscr{W}_1\times \mathscr{W}_2$ such that:
\begin{enumerate}
		\item[{\rm (i)}] $(\l(0),w(0),v(0))=\left(\Phi(\mu),0,\t_{[\mf{L}_2,\mu,d]}\right)$.
		\item[{\rm (ii)}] $\mf{F}(\l(s),\mu,\e,w(s),v(s))=0$ for all $s\in (-\d,\d)$.
		\item[{\rm (iii)}] $v(s)\gg_{2} 0$ if $s\in (-\d,\d)$, $w(s)\gg_1 0$ if $s\in (0,\d)$, and $w(s)\ll_1 0$ if $s\in (-\d,0)$.
		\item[{\rm (iv)}] The set of solutions of \eqref{1.7} in a neighborhood of
$(\l,w,v)=\left(\Phi(\mu),0,\t_{[\mf{L}_2,\mu,d]}\right)$
consists of  the curves $\left(\l,0,\t_{[\mf{L}_2,\mu,d]}\right)$, $\l \thicksim \Phi(\mu)$, and $(\l(s),w(s),v(s))$, $s \in (-\d,\d)$.
\end{enumerate}
Moreover, there are two functions $w_1, w_1^* \gg_1 0$ such that
\begin{equation}
\label{3.9}
	\l'(\e)\!\equiv\!
	\frac{\p\l}{\p s}(0,\e)\!=\!\!\! \int_\O\!\! \left( \e a\!-\! b\t_{[\mf{L}_2,\mu,d]}\right) w_1^2 w_1^*\! +\!\!\!	
		\int_{\Omega}\!\! b  \left(\mathfrak{L}_2\!+\! 2d\t_{[\mf{L}_2,\mu,d]}\!-\!\mu\right)^{-1}\!\!
		\left(\e c\t_{[\mf{L}_2,\mu,d]}w_1\right) w_1w_1^*.
\end{equation}
\end{theorem}
\begin{proof}
By definition, $\mf{F}(\l,\mu,\e,0,\t_{[\mf{L}_2,\mu,d]})=0$. Moreover, the Fr\'{e}ch\`et differential
\[
	\mathscr{L}(\l,\e):= D_{(w,v)}\mf{F}(\l,\mu,\e,0,\t_{[\mf{L}_2,\mu,d]})
\]
is the operator defined by
\[
	\mathscr{L}(\l,\e)(w,v)= \left( \begin{array}{c}
		w-(\mf{L}_1+e)^{-1}\left[(\l+e)w-b \t_{[\mf{L}_2,\mu,d]} w \right] \\[7pt]   v-(\mf{L}_2+e)^{-1}\left[(\mu+e)v -2d \t_{[\mf{L}_2,\mu,d]}v +\e c \t_{[\mf{L}_2,\mu,d]} w \right]   \end{array} \right).
\]
Thus, at $\l=\Phi(\mu)$ we have that $(w_1,v_1)\in N[\mathscr{L}(\Phi(\mu),\e)]$ if, and only if,
\begin{equation}
\label{3.10}
		\left\{
		\begin{array}{ll}
			\left( \mathfrak{L}_1+b  \t_{[\mf{L}_2,\mu,d]}\right)w_1=\Phi(\mu) w_1,
			\\[7pt] \left(\mathfrak{L}_2+2d\t_{[\mf{L}_2,\mu,d]} -\mu \right) v_1 = \e c \t_{[\mf{L}_2,\mu,d]} w_1, \end{array}\right.
\end{equation}
in $\O$ and $\mf{B}_1w_1=\mf{B}_2 v_1=0$ in $\p\O$.
Since $\Phi(\mu)\equiv \s_0\left[ \mf{L}_1+b \t_{[\mf{L}_2,\mu,d]},\mf{B}_1,\O\right]$, by the simplicity
of $\Phi(\mu)$, $w_1$ is unique, up to multiplicative constants. Actually, it can be chosen
to satisfy $w_1\gg_1 0$. Moreover, 	by \eqref{3.3} and Theorem \ref{th2.2}, the second equation of \eqref{3.10} implies that
\begin{equation}
\label{3.11}
		v_1=\left(\mathfrak{L}_2+2d\t_{[\mf{L}_2,\mu,d]} -\mu \right)^{-1} \left(\e c \t_{[\mf{L}_2,\mu,d]} w_1\right).
\end{equation}
Note that $\mathrm{sign\,} v_1 =\mathrm{sign\,} \e$, by Theorem \ref{th2.2}. Therefore,
\[
	N[\mathscr{L}(\Phi(\mu),\e)]=\mathrm{span}[\v_0],\qquad \v_0\equiv (w_1,v_1),\quad w_1\gg_1 0.
\]
Subsequently, we normalize $w_1\gg_1 0$ so that $\int_\O w_1^2(x)\,dx =1$, and denote by
$D_\l \mathscr{L}(\l,\e)$ the derivative of $\mathscr{L}(\l,\e)$ with respect to $\l$. Then,
\begin{equation}
\label{3.12}
	D_\l \mathscr{L}(\Phi(\mu),\e)\v_0 = \left( \begin{array}{c} -(\mf{L}_1+e)^{-1} w_1  \\[7pt] 0 \end{array}\right)\notin 	R[\mathscr{L}(\Phi(\mu),\e)],
\end{equation}
i.e., the transversality condition of Crandall and Rabinowitz \cite{Rab-71b} holds. Indeed,
arguing by contradiction, assume that there exists $(w,v)$ such that
\[
	\mathscr{L}(\Phi(\mu),\e)(w,v)= \left( \begin{array}{c} -(\mf{L}_1+e)^{-1} w_1 \\[7pt] 0 \end{array}\right).
\]
Then, we find from the first equation of this system that
$$
	\left\{ \begin{array}{ll} \left[ \mf{L}_1+ b \t_{[\mf{L}_2,\mu,d]}  -\Phi(\mu)\right]w = -w_1& \quad
\hbox{in}\;\; \O,\\[1ex] \mf{B}_1 w =0 &\quad \hbox{on}\;\;\p\O. \end{array}\right.
$$
By Corollary 7.1(f) of \cite{LG13} this is impossible. This contradiction shows \eqref{3.12}. Consequently,
the first four assertions of the theorem follow from the main theorem of \cite{Rab-71b}. To complete the proof it remains to show \eqref{3.9}.  Setting
\begin{equation*}
(\l(s),w(s),v(s))= \left(\Phi(\mu) +\sum_{j=1}^{\infty}s^j\l_j,
\sum_{j=1}^{\infty}s^jw_j,\t_{[\mf{L}_2,\mu,d]} + \sum_{j=1}^{\infty}s^jv_j\right),\quad s\thicksim 0,
\end{equation*}
and substituting into \eqref{1.7}, it becomes apparent that
\begin{equation}
\label{3.13}
		\left\{ \begin{array}{ll} \left( \mathfrak{L}_1+b\t_{[\mf{L}_2,\mu,d]}-\Phi(\mu)\right) w_2
		=\left(\l_1+b\t_{[\mf{L}_2,\mu,d]}w_1-\e a w_1 -bv_1 \right)w_1 & \quad \hbox{in}\;\;\O,\\[1ex]
  \mf{B}_1 w_2=0 & \quad \hbox{on}\;\; \p\O. \end{array}\right.
\end{equation}
Thanks to Corollary 7.1(e) of \cite{LG13}, it is easily seen that the $L^2(\O)$-orthogonal
to the kernel of the adjoint problem of  $\left( \mathfrak{L}_1+b\t_{[\mf{L}_2,\mu,d]}-\Phi(\mu),\mf{B}_1,\O\right)$ is generated by some
$w_1^*\gg_1 0$. Therefore, multiplying by $w_1^*$ the problem \eqref{3.13} and integrating in $\O$,
it follows from \eqref{3.11} that \eqref{3.9} holds. This ends the proof.
\end{proof}
\par

\begin{remark}
\label{re3.2}
\rm As the dependence of $\mf{F}$ on $\e\in\R$ is also analytic, by the implicit function theorem
used in the proof of the theorem of Crandall and Rabinowitz \cite{Rab-71b}, it becomes apparent that the bifurcated curve
$$
  (\l(s),w(s),v(s))\equiv (\l(s,\e),w(s,\e),v(s,\e))
$$
also is analytic with respect to
the parameter $\e$, though in Theorem \ref{th3.2} we have refrained to emphasize this dependence on the parameter $\e$ to simplify the notations as much as possible.
\end{remark}

As a further application of the exchange stability principle of Crandall and Rabinowitz \cite[Th. 1.16]{CR73}, the next result holds.

\begin{theorem}
\label{th3.3}
	The curve of coexistence states of \eqref{1.7} emanating from $\left(\Phi(\mu),0,\t_{[\mf{L}_2,\mu,d]}\right)$, denoted in Theorem \ref{th3.2} by $(\l(s),\mu,w(s),v(s))$ for sufficiently small  $s > 0$,  is unstable, with one-dimensional unstable manifold,  if $\l'(0)<0$,
and  exponentially stable if $\l'(0)>0$.
\end{theorem}
\begin{proof}
According to \eqref{eig}, it is apparent that
\begin{equation}
\label{3.14}
	\left\{
	\begin{array}{ll}
		\tau_0>0\;\;\hbox{and}\;\;\tau_j>0\;\;\hbox{for all}\; j\geq1\;\; \hbox{if}\;\,\l<\Phi(\mu),\\
		\tau_0=0\;\;\hbox{and}\;\;\tau_j>0\;\;\hbox{for all}\; j\geq1\;\;\hbox{if}\;\,\l=\Phi(\mu),\\
		\tau_0<0\;\;\hbox{and}\;\;\tau_j>0\;\;\hbox{for all}\; j\geq1\;\;\hbox{if}\;\,\l\gtrsim\Phi(\mu).
	\end{array}
	\right.
\end{equation}
Thus, by the exchange stability principle, \cite[Th. 1.16]{CR73}, $(\l(s),\mu,w(s),v(s))$ is linearly unstable (resp. stable) for sufficiently small $s>0$ if $\l'(0)<0$ (resp. $\l'(0)>0$). Moreover, maintaining the notations of the proof of Theorem \ref{th3.2}, it follows from \cite[Sec. 2.4]{LG01} that, for sufficiently small $s>0$,
\begin{equation*}
		\mf{m}\left[D_{(w,v)}\mf{F}(\l(s),\mu,\e,w(s),v(s))\right]=
		\left\{
		\begin{array}{ll}
			\mf{m}\left[\mathscr{L}(\Phi(\mu),\e)\right]+1&\quad\hbox{if}\;\,\l'(0)<0,\\[5pt]
			\mf{m}\left[\mathscr{L}(\Phi(\mu),\e)\right]&\quad\hbox{if}\;\,\l'(0)>0,
		\end{array}
		\right.
\end{equation*}
where $\mf{m}(L)$ stands for the sum of  the algebraic multiplicities of the real negative eigenvalues of $L$. 	 By \eqref{3.14}, $\mf{m}\left[\mathscr{L}(\Phi(\mu),\e)\right]=0$. Therefore,
\begin{equation*}
		\mf{m}\left[D_{(w,v)}\mf{F}(\l(s),\mu,\e,w(s),v(s))\right]=
		\left\{
		\begin{array}{ll}
			1\quad\hbox{if}\;\,\l'(0)<0,\\[5pt]
			0\quad\hbox{if}\;\,\l'(0)>0.	\end{array} \right.
\end{equation*}
The principle of linearized stability of Lyapunov ends the proof.
\end{proof}

\noindent Figure \ref{Fig1} sketches the corresponding local bifurcation diagrams in the transcritical case when $\l'(0)\neq 0$, according to the sign of $\l'(0)$. The arcs of analytic curve filled in by exponentially asymptotically stable solutions have been plotted using continuous lines, whereas unstable solutions with one-dimensional unstable manifold are plotted using dashed lines. The $\l$-axis stands for the constant $\l$-curve $(\l,\mu,0,\t_{[\mf{L}_2,\mu,d]})$. Thanks to Theorem \ref{th3.1}, this solution  is linearly unstable if $\l>\Phi(\mu)$ and linearly stable if $\l<\Phi(\mu)$.
\begin{figure}[h!]
\centering
\includegraphics[scale=0.7]{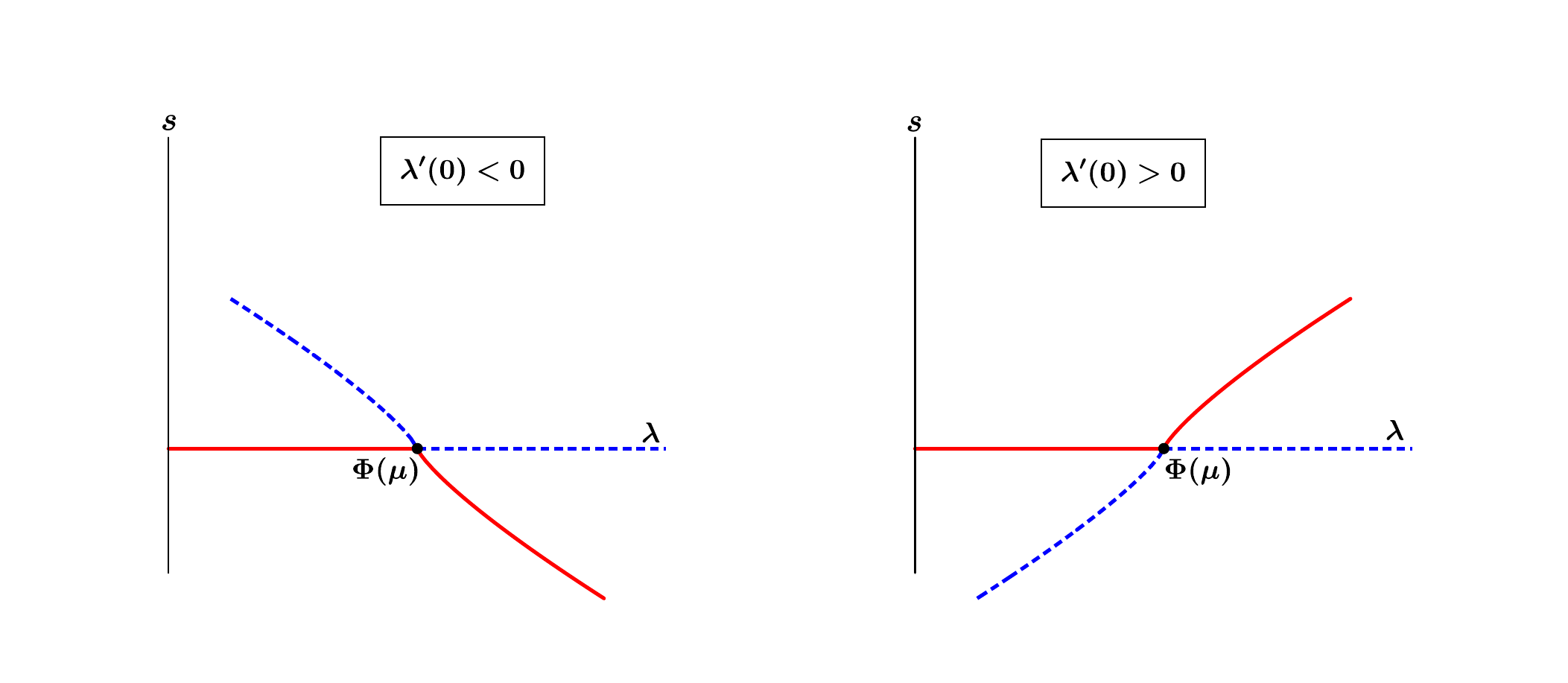}
\caption{Stability of the solutions filling in the bifurcating branches}
\label{Fig1}
\end{figure}

\par
We end this section applying Theorems 4.1 and 5.1 of \cite{LGMH20} to \eqref{1.7}, with $\e>0$. As a direct consequence, the next result holds.

\begin{theorem}
\label{th3.4}
Suppose  \eqref{1.7}, with $\e>0$,  has a coexistence state, $(w,v)$. Then,
\begin{equation}
\label{3.15}
\begin{split}
\l & >\v_\e(\mu)\equiv \s_0\Big[ \mf{L}_1+b \tfrac{\t_{[\mf{L}_2,\mu,d]}}{1+m\t_{[\mf{L}_1,\l,\e a]}},\mf{B}_1,\O\Big],\\[1ex]
\mu & >\Psi_\e(\l)\equiv \s_0\Big[ \mf{L}_2-\e c \tfrac{\t_{[\mf{L}_1,\l,\e a]}}{1+m\t_{[\mf{L}_1,\l,\e a]}},\mf{B}_2,\O\Big].
\end{split}
\end{equation}
Conversely, under the following condition
\begin{equation}
\label{3.16}
\l>\Phi(\mu)\equiv \s_0\left[ \mf{L}_1+b\t_{[\mf{L}_2,\mu,d]},\mf{B}_1,\O\right]\quad\hbox{and}\quad
\mu>\Psi_{\e}(\l),
\end{equation}
the problem \eqref{1.7} has, at least, a coexistence state.
\end{theorem}

Since $\t_{[\mf{L}_2,\mu,d]}= 0$ if $\mu\leq\s_{0,2}$, under this  condition, both \eqref{3.15} and
\eqref{3.16} become into
\begin{equation}
\label{3.17}
\l>\s_{0,1}\quad\hbox{and}\quad \mu>\Psi_\e(\l).
\end{equation}
Therefore, \eqref{3.17} is not only necessary but also sufficient for the existence of a coexistence state if $\mu\leq \s_{0,2}$. Figure \ref{Fig2} sketches the construction of the  wedges \eqref{3.15} and \eqref{3.16} given by
Theorem \ref{th3.4}. Note that, according to Theorem \ref{th2.1},
$$
  \v_\e(\mu)\equiv \s_0\Big[ \mf{L}_1+b \tfrac{\t_{[\mf{L}_2,\mu,d]}}{1+m\t_{[\mf{L}_1,\l,\e a]}},\mf{B}_1,\O\Big]< \s_0[ \mf{L}_1+b \t_{[\mf{L}_2,\mu,d]},\mf{B}_1,\O\Big]\equiv \Phi(\mu)
$$
for all $\mu>\s_{0,2}$. More precisely, by Theorem \ref{th3.4}, \eqref{1.7} has a coexistence state in the solid (dark) area of Figure \ref{Fig2}, whereas outside the union of the solid and dashed patches of Figure \ref{Fig2}, it cannot admit any coexistence state.

\begin{figure}[h!]
\centering
\includegraphics[scale=0.91]{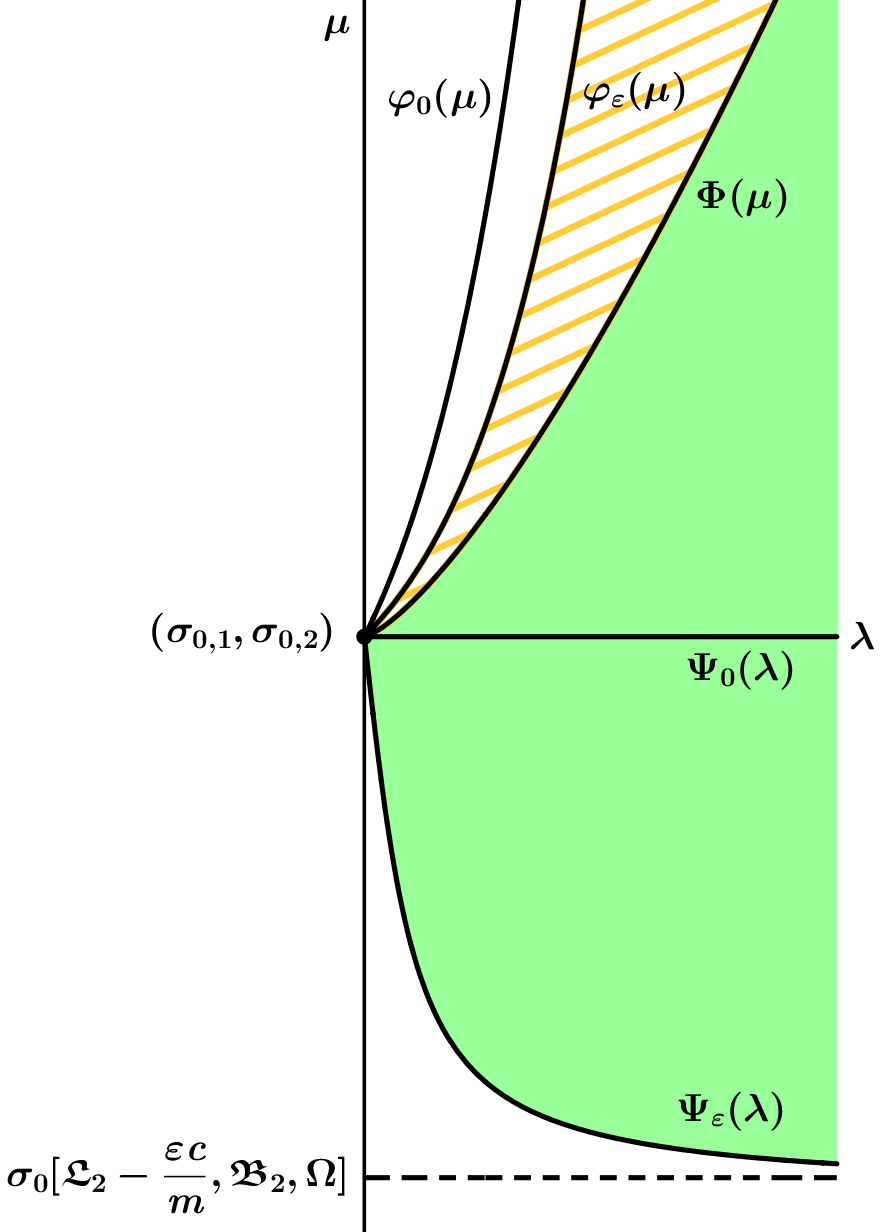}
\includegraphics[scale=0.912]{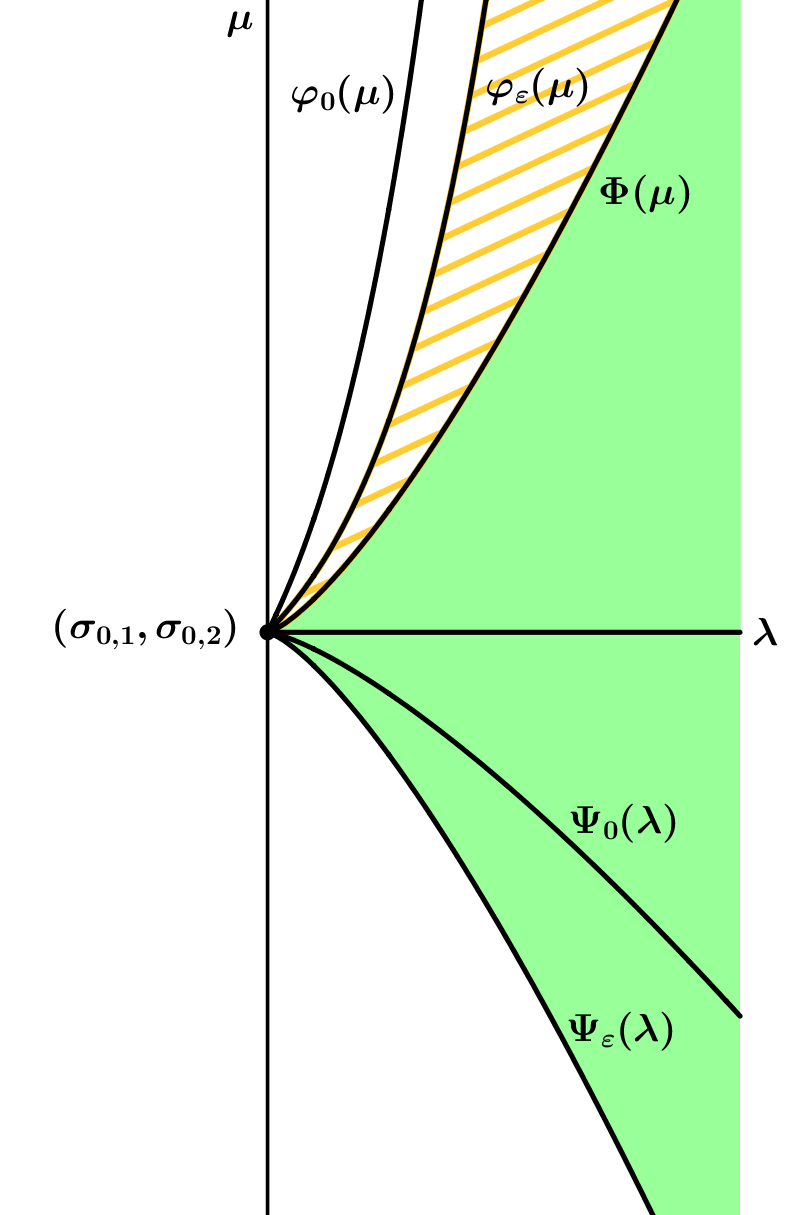}
\caption{The coexistence regions of \eqref{1.7} according to Theorem \ref{th3.4}}
\label{Fig2}
\end{figure}

As already discussed by the authors in Section 3 of \cite{LGMH20}, the first picture of Figure \ref{Fig2}
sketches the behavior of the curve $\mu =\Psi_\e(\l)$, $\l>\s_{0,1}$, when  $m(x)>0$ for all $x\in\bar{\O}$,
whereas the second picture shows it when $\O_0:= \mathrm{int\,}m^{-1}(0)$ is non-empty, which is the
general case dealt with in this paper. In the classical Holling--Tanner case when $m(x)>0$ for all $x\in\bar\O$, thanks to Theorem \ref{th2.2}, it becomes apparent that
$$
\Psi_\e(\l)\equiv \s_0\Big[ \mf{L}_2-\e c\tfrac{m \t_{[\mf{L}_1,\l,\e a]}}{m(1+m\t_{[\mf{L}_1,\l,\e a]})},\mf{B}_2,\O\Big]>\s_0\Big[ \mf{L}_2- \frac{\e c}{m},\mf{B}_2,\O\Big] \quad \hbox{for all}\;\;\l >\s_{0,1},
$$
as illustrated in the first picture of Figure \ref{Fig2}. However, when $\O_0:= \mathrm{int\,}m^{-1}(0)$ is a nice (non-empty) open subset with $\bar\O_0\subset \O$, by \cite[Pr. 3.2]{CCLG}, we have that
$$
  \Psi_\e(\l)\equiv \s_0\Big[ \mf{L}_2-\e c  \tfrac{\t_{[\mf{L}_1,\l,\e a]}}{1+m\t_{[\mf{L}_1,\l,\e a]}},\mf{B}_2,\O\Big] < \s_0\Big[ \mf{L}_2-\e c \t_{[\mf{L}_1,\l,\e a]},\mf{D},\O_0\Big],
$$
where $\mf{D}$ stands for the Dirichlet boundary operator on $\p\O_0$. Thus,
$$
  \lim_{\l\ua \infty} \Psi_\e(\l)=-\infty \quad \hbox{for all}\;\;\e>0,
$$
as illustrated in the second picture of Figure \ref{Fig2}, which is a behavior reminiscent of the
one exhibited by the classical Lotka--Volterra model.
\par
Since
\begin{equation}
\label{3.18}
   \t_{[\mf{L}_1,\l,\e a]} =\e^{-1} \t_{[\mf{L}_1,\l,a]},
\end{equation}
it is apparent that
\begin{align*}
\lim_{\e\da 0}\v_\e(\mu) & =\lim_{\e\da 0} \s_0\Big[ \mf{L}_1+b \tfrac{\t_{[\mf{L}_2,\mu,d]}}{1+\frac{m}{\e}\t_{[\mf{L}_1,\l,a]}},\mf{B}_1,\O\Big]\\
& = \s_0\left[\mf{L}_1+\left(1-\chi_{_{\mathrm{int \, supp}\,m}}\right)
b(x)\t_{[\mf{L}_2,\mu,d]},\mf{B}_1,\O\right],
\end{align*}
where, for any subset $A\subset\R^N$, $\chi_{_A}$ stands for the characteristic function of
the set $A$, i.e., $\chi_{_A}(x)=1$ if $x\in A$, and $\chi_{_A}(x)=0$ if $x\in\R^N\setminus A$.
In the next section, it will become apparent that the function
\begin{equation}
\label{3.19}
  \v_0(\mu):= \s_0\left[\mf{L}_1+\left(1-\chi_{_{\mathrm{int\,supp}\,m}}\right)
b(x)\t_{[\mf{L}_2,\mu,d]},\mf{B}_1,\O\right],\qquad \mu>\s_{0,2},
\end{equation}
provides us with the left limiting curve to the region  where the uncoupled model \eqref{1.8} possesses a coexistence state; recall that \eqref{1.8} is \eqref{1.7} with $\e=0$. The curve $\l=\v_0(\mu)$ has been also plotted in Figure \ref{Fig2}. According to Theorem \ref{th2.1}, since
$$
  1-\chi_{_{\mathrm{int\,supp}\,m}} \lneq \tfrac{1}{1+m\t_{[\mf{L}_1,\l,\e a]}}\quad \hbox{in}\;\; \O,
$$
it follows that, for every $\mu>\s_{0,2}$ and $\e>0$,
\begin{equation}
\label{3.20}
  \v_0(\mu)<\v_\e(\mu),
\end{equation}
provided $bm\gneq0$, as illustrated in Figure \ref{Fig2}. Finally, note that, for every $\l>\s_{0,1}$,
\begin{align*}
\lim_{\e\da 0}\Psi_\e(\l) & =\lim_{\e\da 0} \s_0\Big[ \mf{L}_2-c \tfrac{\t_{[\mf{L}_1,\l,a]}}{1+\frac{m}{\e}\t_{[\mf{L}_1,\l,a]}},\mf{B}_2,\O\Big]\\
& = \s_0\left[\mf{L}_2-\left(1-\chi_{_{\mathrm{int \, supp}\,m}}\right)c(x)\t_{[\mf{L}_1,\l,a]},\mf{B}_2,\O\right]\equiv\Psi_0(\l)\leq\s_{0,2}.
\end{align*}
Although $\Psi_0(\mu)$ can take different values depending on the distribution of the patches where $m(x)$ and $c(x)$ vanish, this does not affect the analysis of \eqref{1.8}, for as the condition $\mu>\s_{0,2}$ is necessary for the existence of coexistence states.

\section{The coexistence states of the limiting system \eqref{1.8}}
\label{sec4}
This section determines the set of coexistence states of the limiting shadow problem \eqref{1.8}.
Since the component $v$ satisfies
\begin{equation*}
\left\{
\begin{array}{lll}
\mf{L}_2 v=\mu v -d(x)v^2&\quad \hbox{in}\;\;\Omega,\\[1ex]
\mf{B}_2 v=0&\quad\hbox{on}\;\;\partial\Omega,
\end{array}
\right.	
\end{equation*}
the condition $\mu>\s_{0,2}\equiv \s_0[\mf{L}_2,\mf{B}_2,\O]$ is imperative so that
\eqref{1.8} can admit a coexistence state. Otherwise, $v=0$ for any component-wise nonnegative
solution, $(w,v)$, of \eqref{1.8}. Thus, throughout this section, we assume that
$\mu>\s_{0,2}$. In such case, by Theorem \ref{th2.3}, for every coexistence state $(w,v)$ of \eqref{1.8},
necessarily $v=\t_{[\mf{L}_2,\mu,d]}\gg_2 0$, and $w\gg_1 0$ is a positive solution of the associated problem
\begin{equation}
\label{4.1}
\left\{
\begin{array}{lll}
\mf{L}_1 w=\lambda w-b(x)\t_{[\mf{L}_2,\mu,d]}\dfrac{w}{1+m(x)w} &\quad \hbox{in}\;\;\Omega,\\[10pt]
\mf{B}_1 w=0 &\quad\hbox{on}\;\;\partial\Omega.
\end{array}
\right.
\end{equation}
Note that, as soon as  $b(x)$ and $m(x)$ have disjoint supports, i.e., $bm=0$, one has that
$$
\left(  1-\chi_{_{\mathrm{int\,supp}\,m}} \right) b = \frac{b}{1+m w}
$$
and, hence, \eqref{4.1} becomes into the linear problem
\begin{equation*}
\left\{
\begin{array}{lll}
\left[\mf{L}_1+\left(  1-\chi_{_{\mathrm{int\,supp}\,m}} \right) b \t_{[\mf{L}_2,\mu,d]}\right] w=\lambda w&\quad \hbox{in}\;\;\Omega,\\[10pt]
	\mf{B}_1 w=0 &\quad\hbox{on}\;\;\partial\Omega.
\end{array}
\right.
\end{equation*}
Therefore, when $bm=0$, \eqref{4.1} has a positive solution if, and only of, $\l=\v_0(\mu)$ (see \eqref{3.19}) and, in such case, $w$ is a positive solution if, and only if, $w=s w_0$ for some $s>0$, where $w_0\gg_1 0$ stands for
any principal eigenfunction associated to $\l=\v_0(\mu)$. The next result collects some useful properties of \eqref{4.1}

\begin{lemma}
\label{le4.1}
Suppose $w\neq 0$ is a positive solution of \eqref{4.1}. Then, $w\gg_1 0$ and
\begin{equation}
\label{4.2}
\l=\s_0\Big[\mf{L}_1+\tfrac{b(x)\t_{[\mf{L}_2,\mu,d]}}{1+m(x)w},\mf{B}_1,\O\Big].
\end{equation}
Thus,
\begin{equation}
\label{4.3}
\s_{0,1}\leq\varphi_0(\mu)\leq \l <\Phi(\mu),
\end{equation}
where $\v_0(\mu)$ and $\Phi(\mu)$ are the functions defined in \eqref{3.19} and \eqref{3.16}, respectively.
More precisely,
\begin{equation*}
\left\{
\begin{array}{ll}
\s_{0,1}\leq\varphi_0(\mu)< \l <\Phi(\mu)&\qquad\hbox{if}\;\; bm\gneq0,\\[1ex]
\s_{0,1}<\varphi_0(\mu)=\l <\Phi(\mu)&\qquad\hbox{if}\;\; bm=0.
\end{array}
\right.
\end{equation*}
In other words, either $(\l,\mu)$ lies in the wedges region between the curves $\v_0(\mu)$ and $\Phi(\mu)$ in Figure \ref{Fig2} if $bm\gneq0$, or $\l=\v_0(\mu)$ if $bm=0$.
\end{lemma}

\begin{proof}
Since
$$
   \left( \mf{L}_1+\tfrac{b(x)\t_{[\mf{L}_2,\mu,d]}}{1+m(x)w}\right)w =\l w \quad \hbox{in}\;\;\O,
$$
and $\mf{B}_1 w=0$, with $w \gneq 0$, the identity \eqref{4.2} is a direct consequence of the uniqueness
of $\s_0$, and $w\gg_1 0$, by the properties of the positive principal eigenfunctions.
\par
On the other hand, by our assumptions on $m(x)$,
$$
   1-\chi_{_{\mathrm{int\,supp}\,m}} \lneq \frac{1}{1+m w}\quad \hbox{in}\;\; \O.
$$
Thus, since $b\t_{[\mf{L}_2,\mu,d]}\gneq 0$, it is apparent that
\begin{equation}
\label{4.4}
0\leq \left(  1-\chi_{_{\mathrm{int\,supp}\,m}} \right) b \t_{[\mf{L}_2,\mu,d]}\leq \frac{b\t_{[\mf{L}_2,\mu,d]}}{1+m w}\lneq b \t_{[\mf{L}_2,\mu,d]}\quad \hbox{in}\;\;\O.
\end{equation}
Therefore, by \eqref{4.4} and  Theorem \ref{th2.1}, \eqref{4.3} holds.
\par
Now, note that $\v_0(\mu)= \l$ unless $b(x)>0$ for some $x\in \mathrm{supp\,}m$. Thus, $\v_0(\mu)<\l$ if
$bm\gneq 0$,  and $\v_0(\mu)=\l>\s_{0,1}$ if $bm= 0$. Moreover, in  case $bm\gneq 0$, we have that, for every $\mu>\s_{0,2}$,
$$
\v_0(\mu)= \s_0\left[\mf{L}_1+\left(1-\chi_{_{\mathrm{int\,supp}\,m}}\right)
b\t_{[\mf{L}_2,\mu,d]},\mf{B}_1,\O\right] = \s_{0,1}
$$
if, and only if,
$$
  \left(1-\chi_{_{\mathrm{int \,supp}\,m}}\right) b=0,
$$
and this occurs provided $m(x)>0$ for all $x\in\O$, like in the
Holling--Tanner model, or $m^{-1}(0)\subset b^{-1}(0)\neq\emptyset$, i.e., $\supp\,b\subset \supp\,m\neq\emptyset$. This ends the proof.
\end{proof}

It should not be forgotten that
throughout this paper we are assuming that $m=0$ in $\bar \O_0$ but $m\gneq 0$. Therefore, we are in a rather hybrid (different) situation between the Lotka--Volterra and the Holling--Tanner models.
\par
Throughout the rest of this paper, we will assume that $bm\gneq 0$, even if not expressed within text.
This entails that $\v_0(\mu)<\l<\Phi(\mu)$ if \eqref{4.1} has some positive solution. The next results collects two important qualitative properties of the positive solutions of \eqref{4.1}.

\begin{lemma}
\label{le4.2}
Let $\{(\l_n,w_n)\}_{n\geq1}$ be a sequence of positive solutions of \eqref{4.1} such that
$$
  \lim_{n\to+\infty}\l_n=\l_*.
$$
Then $\l^* \in[\v_0(\mu),\Phi(\mu)]$. Moreover:
\begin{enumerate}
\item[\rm (a)] $\displaystyle\lim_{n\to\infty}\|w_n\|_{\infty}=+\infty$ if, and only if, $\l_*=\v_0(\mu)$;
\item[\rm (b)]$\displaystyle \lim_{n\to\infty}\|w_n\|_{\infty}=0$ if, and only if, $\l_*=\Phi(\mu)$.
\end{enumerate}
\end{lemma}
\begin{proof}
By \eqref{4.3}, necessarily,
\begin{equation}
\label{iv.5}
  \varphi_0(\mu)< \l_n <\Phi(\mu)\quad \hbox{for all}\;\; n\geq 1.
\end{equation}
Thus, letting $n\to\infty$, yields
to  $\l^*\in [\v_0(\mu),\Phi(\mu)]$. Now, in order to prove the necessity of Part (a), suppose that
\begin{equation}
\label{iv.6}
  \lim_{n\to+\infty}\|w_n\|_{\infty}=+\infty.
\end{equation}
Then, by expressing \eqref{4.1} as a fixed point equation and dividing by $\|w_n\|_{\infty}$, it becomes apparent that, for every  $e>\s_{0,1}$ and $n\geq 1$,
\begin{equation}
\label{iv.7}
\frac{w_n}{\|w_n\|_{\infty}}=(\mf{L}_1+e)^{-1}\Big[(\l_n+e)\frac{w_n}{\|w_n\|_{\infty}}-
b\t_{[\mf{L}_2,\mu,d]}\tfrac{w_n}{\|w_n\|_{\infty}(1+m\frac{w_n}{\|w_n\|_{\infty}}\|w_n\|_{\infty})}\Big].	 \end{equation}
Since the sequence of continuous functions
\[
(\l_n+e)\frac{w_n}{\|w_n\|_{\infty}}-b\t_{[\mf{L}_2,\mu,d]}\frac{w_n}{\|w_n\|_{\infty}(1+m\frac{w_n}{\|w_n\|_{\infty}}
\|w_n\|_{\infty})},\qquad n\geq 1,
\]
is bounded in $\mc{C}(\bar\O)$ and  $(\mf{L}_1+e)^{-1}$ is a compact operator, it follows from
\eqref{iv.7} that there exists
$\psi\in \mc{C}^1_{\mf{B}_1}(\bar\O)$ such that, along some subsequence, labeled by $n_\ell$,
\begin{equation}
\label{iv.8}
 \lim_{\ell\to+\infty}\frac{w_{n_\ell}}{\|w_{n_\ell}\|_{\infty}}=\psi\quad \hbox{in}\;\; \mc{C}^1_{\mf{B}_1}(\bar\O).
\end{equation}
By \eqref{iv.8}, $\psi\gneq0$ and $\|\psi\|_{\infty}=1$. Moreover, by elliptic regularity,
particularizing \eqref{iv.7} at $n=n_\ell$ and letting $\ell\to+\infty$ in the resulting identity,
shows that $\psi\in\mathscr{W}_1$ and that it solves the problem
\begin{equation}
\label{iv.9}
\left\{
\begin{array}{ll}
\left[\mf{L}_1+(1-\chi_{_{\mathrm{int\,supp\,}m}})b\t_{[\mf{L}_2,\mu,d]} \right]\psi=\l_*\psi &\quad \hbox{in}\;\;\O,\\[1ex]
\mf{B}_1\psi=0 & \quad \hbox{on}\;\; \p\O.
\end{array}
\right.
\end{equation}
From \eqref{iv.9} it is apparent that  $\psi\gg_1 0$ and that
$$
   \l_*=\s_0\left[\mf{L}_1+(1-\chi_{_{\mathrm{int\,supp\,}m}})b\t_{[\mf{L}_2,\mu,d]},\mf{B}_1,\O\right]
   \equiv \v_0(\mu).
$$
This shows that, indeed, $\l_*=\v_0(\mu)$ if \eqref{iv.6} holds.
\par
Adapting the previous argument, it is easily seen that
\begin{equation}
\label{iv.10}
  \lim_{n\to+\infty}\|w_n\|_{\infty}=0
\end{equation}
guarantees the existence of some $\psi\in\mathscr{W}_1$, with $\|\psi\|_\infty=1$, such that
\begin{equation*}
\left\{
\begin{array}{ll}
\left(\mf{L}_1+b\t_{[\mf{L}_2,\mu,d]} \right)\psi=\l_*\psi &\quad \hbox{in}\;\;\O,\\[1ex]
\mf{B}_1\psi=0 &\quad \hbox{on}\;\; \p\O. \end{array} \right.
\end{equation*}
Therefore, \eqref{iv.10} implies  that
\begin{equation*}
  \l_* = \s_0[\mf{L}_1+b\t_{[\mf{L}_2,\mu,d]},\mf{B}_1,\O]\equiv \Phi(\mu),
\end{equation*}
which ends the proof of the necessity in Part (b).
\par
To prove the sufficiency in Part (a), assume that
$\l_*=\v_0(\mu)$ and that \eqref{iv.6} fails. Then, there exists a constant, $C>0$, such that, along some subsequence of $\{w_n\}_{n\geq 1}$,
\begin{equation}
\label{4.11}
  \|w_{n_\ell}\|_\infty \leq C \quad \hbox{for all}\;\; \ell\geq 1.
\end{equation}
By the necessity of Part (b), $\{w_{n_\ell}\}_{n\geq 1}$ cannot admit any subsequence converging to zero in $\mc{C}(\bar\O)$, because, in such case, $\v_0(\mu)=\l_*=\Phi(\mu)$, which contradicts $\v_0(\mu)<\Phi(\mu)$
(see \eqref{iv.5}). On the other hand, since
\begin{equation}
\label{4.12}
  w_{n_\ell}=(\mf{L}_1+e)^{-1}\left[(\l_{n_\ell}+e)w_{n_\ell}-b\t_{[\mf{L}_2,\mu,d]} \frac{w_{n_\ell}}{1+mw_{n_\ell}}\right] \quad \hbox{for all}\;\; \ell\geq 1
\end{equation}
and, due to \eqref{4.11}, the sequence
$$
  (\l_{n_\ell}+e)w_{n_\ell}-b\t_{[\mf{L}_2,\mu,d]} \frac{w_{n_\ell}}{1+mw_{n_\ell}},\qquad \ell\geq 1,
$$
is bounded, by the compactness of $(\mf{L}_1+e)^{-1}$, we can extract a subsequence of  $\{w_{n_\ell}\}_{n\geq 1}$, relabeled by $n_\ell$, such that, for some $\Psi\in \mc{C}^1_{\mf{B}_1}(\bar\O)$,
\begin{equation}
\label{4.13}
  \lim_{\ell\to\infty} w_{n_\ell} =\Psi \quad \hbox{in}\;\; \mc{C}_{\mf{B}_1}^1(\bar\O).
\end{equation}
As we already know that $\{w_{n_\ell}\}_{n\geq 1}$ cannot converge to zero in $\mc{C}(\bar\Omega)$, it becomes apparent that $\Psi\gneq 0$. Moreover, letting $\ell\to\infty$ in \eqref{4.12} shows that
\begin{equation}
\label{4.14}
  \Psi =(\mf{L}_1+e)^{-1}\left[(\v_0(\mu)+e)\Psi -b\t_{[\mf{L}_2,\mu,d]} \frac{\Psi}{1+m\Psi}\right]   \quad \hbox{for all}\;\; \ell\geq 1.
\end{equation}
By elliptic regularity, it follows from \eqref{4.14} that $\Psi\in\mathscr{W}_1$ and that it provides us with a positive solution of
\begin{equation*}
\left\{
\begin{array}{ll}
\left(\mf{L}_1+\frac{b\t_{[\mf{L}_2,\mu,d]}}{1+m\Psi}\right)\Psi=\v_0(\mu)\Psi &\quad \hbox{in}\;\;\O,\\[1ex]
\mf{B}_1\Psi =0 &\quad \hbox{on}\;\; \p\O.
\end{array}
\right.
\end{equation*}
Therefore, $\Psi\gg_1 0$ and, by the uniqueness of $\s_0$, it becomes apparent that
\begin{equation}
\label{4.15}
\v_0(\mu)=\s_0\left[\mf{L}_1+\tfrac{b\t_{[\mf{L}_2,\mu,d]}}{1+m\Psi},\mf{B}_1,\O\right].
\end{equation}
On the other hand, since $bm\gneq0$,
$$
  \left(1-\chi_{_{\mathrm{int\,supp}\,m}}\right)b \lneq \tfrac{b}{1+m\Psi}\quad \hbox{in}\;\; \O,
$$
it follows from Theorem \ref{th2.1} and the definition of $\v_0(\mu)$ that, for every $\mu>\s_{0,2}$,
$$
  \v_0(\mu):= \s_0\left[\mf{L}_1+\left(1-\chi_{_{\mathrm{int\,supp}\,m}}\right)
b\t_{[\mf{L}_2,\mu,d]},\mf{B}_1,\O\right] < \s_0\left[\mf{L}_1+\tfrac{b\t_{[\mf{L}_2,\mu,d]}}{1+m\Psi},\mf{B}_1,\O\right].
$$
As this estimate contradicts \eqref{4.15}, \eqref{4.11} fails. Therefore, \eqref{iv.6} holds. This ends the proof of Part (a).
\par
To complete the proof of Part (b), suppose that $\l_*=\Phi(\mu)$. Then, since $\v_0(\mu)<\Phi(\mu)$
for all $\mu>\s_{0,2}$, it follows from Part (a) that there exists a constant $C>0$ such that
\begin{equation}
\label{4.16}
  \|w_{n}\|_\infty \leq C \quad \hbox{for all}\;\; n \geq 1.
\end{equation}
In such case, adapting the previous compactness arguments, it becomes apparent that there exist
$\Psi \in\mathscr{W}_1$, with $\Psi\geq 0$, and a subsequence of $\{w_n\}_{n\geq 1}$,
relabeled by $n_\ell$, $\ell\geq 1$, such that \eqref{4.13} holds. Thus, since $\l_*=\Phi(\mu)$, necessarily
\begin{equation}
\label{4.17}
\left\{
\begin{array}{ll}
\left(\mf{L}_1+\frac{b\t_{[\mf{L}_1,\mu,d]}}{1+m\Psi}\right)\Psi=\Phi(\mu)\Psi & \quad \hbox{in}\;\;\O,\\[1ex]
\mf{B}_1\Psi=0 & \quad \hbox{on}\;\;\p\O. \end{array} \right.
\end{equation}
Suppose that $\Psi\gneq0$. Then, $\Psi\gg_1 0$ and \eqref{4.17} implies that
\begin{equation}
\label{4.18}
  \Phi(\mu)=\s_0\left[\mf{L}_1+\tfrac{b\t_{[\mf{L}_1,\mu,d]}}{1+m\Psi},\mf{B}_1,\O\right].
\end{equation}
On the other hand, by Theorem \ref{th2.2} and the definition of $\Phi(\mu)$, we have that
$$
 \Phi(\mu) \equiv \s_0[ \mf{L}_1+b \t_{[\mf{L}_2,\mu,d]},\mf{B}_1,\O\Big]>
 \s_0\left[\mf{L}_1+\tfrac{b\t_{[\mf{L}_1,\mu,d]}}{1+m\Psi},\mf{B}_1,\O\right],
$$
which contradicts \eqref{4.18}. Thus, $\Psi=0$ and hence,
\begin{equation}
\label{4.19}
  \lim_{\ell\to\infty} w_{n_\ell} = 0 \quad \hbox{in}\;\; \mc{C}_{\mf{B}_1}^1(\bar\O).
\end{equation}
As this argument can be repeated along any subsequence of $\{w_n\}_{n\geq 1}$,  Part (b) holds.
This ends the proof of the lemma.
\end{proof}

Particularizing \eqref{3.8} at $\e=0$ yields to
\begin{equation}
\label{4.20}
\mf{F}_0(\l,\mu,w,v)\equiv\mf{F}(\l,\mu,0,w,v):=\left( \begin{array}{c}
		w- (\mf{L}_1+e)^{-1}\left[ (\l+e)w- b \frac{wv}{1+mw} \right] \\[7pt]
		v - (\mf{L}_2+e)^{-1} \left[ (\mu+e)v-dv^2 \right] \end{array} \right).
\end{equation}
As the $v$-component component of \eqref{4.20} vanishes at $v=\t_{[\mf{L}_2,\mu,d]}$, it becomes apparent
that $w$ is a positive solution of \eqref{4.1} if, and only if, the $w$-component of
$$
\mf{F}_0(\l,\mu,w,\t_{[\mf{L}_2,\mu,d]}):=\left( \begin{array}{c}
		w- (\mf{L}_1+e)^{-1}\left[ (\l+e)w- b \t_{[\mf{L}_2,\mu,d]} \frac{w}{1+mw} \right] \\[7pt]
		0 \end{array} \right),
$$
vanishes. Naturally, this is also a rather direct consequence of \eqref{4.1}. Therefore, when applying
Theorem \ref{th3.2} to \eqref{1.8} at $\e=0$ it becomes apparent that $v(s) = \t_{[\mf{L}_2,\mu,d]}$
and that $\mf{F}_0(\l(s),\mu,w(s),\t_{[\mf{L}_2,\mu,d]})=0$
for all $s \in(-\d,\d)$. Moreover, particularizing \eqref{3.9}
at $\e=0$ provides us with
\begin{equation}
\label{4.21}
		\l'(0)= - \int_\O b\t_{[\mf{L}_2,\mu,d]} w_1^2 w_1^* <0.
\end{equation}
Therefore, there is a bifurcation to positive solutions of \eqref{4.1} from $(w,v)=(0,\t_{[\mf{L}_2,\mu,d]})$ at $\l=\Phi(\mu)$ and
the bifurcation is subcritical, because of \eqref{4.21}, or \eqref{4.3}. Naturally, this entails the existence of an $\e_0>0$ such that $\l'(0)<0$ in \eqref{1.7} if $|\e|<\e_0$.
\par
Subsequently, we denote by $\mathscr{S}_0$ the set of nontrivial solutions of \eqref{4.1} defined by
$$
\mathscr{S}_0:=\{(\l,\mu,w,\t_{[\mf{L}_2,\mu,d]})\in\mf{F}_0^{-1}(0)\,:\,w\neq0\}
\cup\{(\l,\mu,0,\t_{[\mf{L}_2,\mu,d]})\;:\;\l\in\Sigma(\mathscr{L}(\l,0))\},
$$
where $\Sigma(\mathscr{L}(\l,0))$ stands for the generalized spectrum of the Fredholm curve
$\mathscr{L}(\l,0)$ introduced in the proof of Theorem \ref{th3.2}. And
$\mathscr{C}_0^+$ stands for the subcomponent of positive solutions of $\mathscr{S}_0$ such that
$(\Phi(\mu),\mu,0,\t_{[\mf{L}_2,\mu,d]})\in \bar{\mathscr{C}}_0^+$. The next result provides us with some
useful properties of $\mathscr{C}_0^+$. We are denoting by
$\mc{P}_\l$ the $\l$-projection operator,
$$
  \mc{P}_\l(\l,\mu,w,\t_{[\mf{L}_2,\mu,d]})=\l.
$$

\begin{theorem}
\label{th4.1}
The component $\mathscr{C}_0^+$ satisfies
\begin{equation}
\label{4.22}
  \mc{P}_\l (\mathscr{C}_0^+)=(\v_0(\mu),\Phi(\mu)).
\end{equation}
Moreover, for every sequence of positive solutions
in $\mathscr{C}_0^+$, $\{(\l_n,\mu,w_n,\t_{[\mf{L}_2,\mu,d]})\}_{n\geq 1}$, such that $\lim_{n\to\infty}\l_n=\v_0(\mu)$, necessarily
\begin{equation}
\label{4.23}
 \lim_{n\to \infty} \|w_n\|_\infty =+\infty.
\end{equation}
In other words, $\mathscr{C}_0^+$ is unbounded at $\l=\v_0(\mu)$.
\end{theorem}
\begin{proof}
Owing to Lemma \ref{le4.2}(b), $(\Phi(\mu),\mu,0,\t_{[\mf{L}_2,\mu,d]})$ is the unique bifurcation point to positive solutions from $(\l,\mu,0,\t_{[\mf{L}_2,\mu,d]})$. The existence of $\mathscr{C}_0^+$ follows from Theorem \ref{th3.2} and the Zorn--Kuratowski lemma. By  \cite[Th.7.1.3]{LG01}, $\mathscr{C}_0^+$ is unbounded in $\R^2\times \mc{C}(\bar\O)$. Since $\mu>\s_{0,2}$ is fixed and, due to Lemma \ref{le4.1}, $\l\in (\v_0(\l),\Phi(\mu))$ if  $bm\gneq 0$,
$\mathscr{C}_0^+$ must be unbounded in $w$. Thus, thanks to Lemma \ref{le4.2}(a),
\eqref{4.22} and \eqref{4.23} hold.
\end{proof}

The next result provides us with the fine structure of the component $\mathscr{C}_0^+$ near $\l=\Phi(\mu)$ and $\l=\v_0(\mu)$. It is a pivotal result in getting the  main multiplicity result of this paper for \eqref{1.8}
with sufficiently small $\e>0$.

\begin{theorem}
\label{th4.2}
In a neighborhood of $(\l,\mu,w,\t_{[\mf{L}_2,\mu,d]})=(\Phi(\mu),\mu,0,\t_{[\mf{L}_2,\mu,d]})$ in $\R\times\R\times \mathscr{W}_1\times \{\t_{[\mf{L}_2,\mu,d]}\}$, $\mathscr{C}_0^+$ consists
of the analytic curve $(\l(s),\mu,w(s),\t_{[\mf{L}_2,\mu,d]})$ given by Theorem \ref{th3.2}. Moreover, the following properties are satisfied:
\begin{enumerate}
\item[{\rm (a)}]  For sufficiently small $r>0$ and every $\l \in [\Phi(\mu)-r,\Phi(\mu))$, \eqref{4.1} has a unique positive solution, which is linearly unstable with one-dimensional unstable manifold.
\par
\item[{\rm (b)}]  There exists $r>0$ such that, for every $\l\in (\v_0(\mu),\v_0(\mu)+r]$, \eqref{4.1} has a unique positive solution, $(\l,\mu,w_\l,\t_{[\mf{L}_2,\mu,d]})$, which is non-degenerate. Thus, for these values of $\l$, $\mathscr{C}_0^+$ consists of an analytic curve of positive solutions bifurcating from $+\infty$ at $\l=\v_0(\mu)$, in the sense that
\begin{equation}
\label{4.24}
  \lim_{\l \da \v_0(\mu)}w_\l(x) =+\infty\quad \hbox{for all}\;\; x\in \O.
\end{equation}
\end{enumerate}
Furthermore, these solutions have local Poincar\'{e} index $-1$, calculated through the Leray--Schauder degree.
\end{theorem}
\begin{proof}
According to \eqref{4.21}, $\mathscr{C}_0^+$ bifurcates subcritically from $w=0$ at $\l=\Phi(\mu)$.
Combining this feature together with the uniqueness of the bifurcated curve in Theorem \ref{th3.2} and
Lemma \ref{le4.2} (b), it becomes apparent the existence of a $r >0$ such that \eqref{4.1} has a unique solution for each $\l\in[\Phi(\mu)-r,\Phi(\mu))$. The fact that $\mathscr{C}_0^+$ is analytic for $\l$
sufficiently close to $\Phi(\mu)$ is a byproduct of Theorem \ref{th3.2}, since $\mathscr{C}_0^+$ can be
parameterized by $\l$,  and $\mf{F}$, or  $\mf{F}_0$, is an analytic function of $\l$. Furthermore, since
$\l'(0)<0$, it follows from Theorem \ref{th3.3} that, for sufficiently small $r>0$ and every
$\l\in[\Phi(\mu)-r,\Phi(\mu))$, the positive solution $(\l,\mu,w(\l))$ is linearly unstable with one-dimensional unstable manifold. In particular, by the Schauder formula, its local index as a fixed point of the compact operator $I-\mf{F}_0$ equals $-1$.
\par
On the other hand, by Lemma \ref{le4.2} (a), for every $\l\in(\v_0(\mu),\Phi(\mu))$, there exists $M_{\l}>0$ such that any positive solution, $(\tilde \l,\tilde w)$,  of \eqref{4.1} with  $\tilde \l \in [\l,\Phi(\mu))$ satisfies
$$
    \tilde w \in W_{\l}:=\left\{w\in \mathscr{W}_1 \,:\,0<\|w\|_{\infty}<M_{\l}\right\}.
$$
Thus, combining the homotopy invariance with the excision property of the Leray--Schauder degree, it becomes apparent that
$$
  \mathrm{Deg\,} (\mf{F}_0(\tilde \l,\mu,\cdot),W_{\l})=
  \mathrm{Deg\,} (\mf{F}_0(\Phi(\mu)-\tfrac{r}{2},\mu,\cdot),W_{\l})=-1
$$
for all $\tilde \l\in [\l,\Phi(\mu))$. In particular,
\begin{equation}
\label{4.25}
\mathrm{Deg\,}\,(\mf{F}_0(\l,\mu,\cdot,),W_{\l})=-1\quad\hbox{for all}\;\; \l\in(\v_0(\mu),\Phi(\mu)).
\end{equation}
Therefore, for every $\l\in(\v_0(\mu),\Phi(\mu))$, the total sum of the local Poincar\'{e} indices of the \eqref{4.1} positive solutions, calculated through the Leray--Schauder degree, equals $-1$.
\par
Subsequently,  we will carry out the (sharp) analysis of $\mathscr{C}_0^+$ in a neighborhood of $\l=\v_0(\mu)$. Let $\{(\l_n,w_n)\}_{n\geq1}$ be a sequence of positive solutions of $\mathscr{C}_0^+$ such that
\begin{equation}
\label{4.26}
  \lim_{n\to+\infty}\l_n=\v_0(\mu).
\end{equation}
Then, by Lemma \eqref{le4.2} (a), we already know that
$$
   \lim_{n\to\infty}\|w_n\|_{\infty}=+\infty.
$$
Note that, in particular, this implies that $\lim_{n\to \infty}M_{\l_n}=+\infty$. Moreover, according to
the proof of Lemma \eqref{le4.2} (a),  there exists a subsequence, labeled again by $n$, such that
\[
   \lim_{n\to+\infty}\frac{w_{n}}{||w_{n}||_{\infty}}=\psi
\]
for some $\psi\in\mathscr{W}_1$ solving \eqref{iv.9}. Since $\psi\gg_1 0$ is a principal eigenfunction associated with $\v_0(\mu)$, it becomes apparent that
\begin{equation}
\label{4.27}
\lim_{n\to+\infty}w_n(x)=+\infty\quad\hbox{for all}\;\; x\in\O.
\end{equation}
As this holds for every sequence of positive solutions, once established the uniqueness of $w_\l$, \eqref{4.24} holds. In order to prove the uniqueness of the positive solution
for $\l$ in a right-neighborhood of $\v_0(\mu)$, we will show that, for sufficiently large $n$, $(\l_n,w_n)$ must be non-degenerate with a one-dimensional unstable manifold. Thanks again to the Schauder formula, this entails that the local index of these positive solutions equals $-1$ and therefore, combining \eqref{4.25} with the additivity property of the Leray--Schauder degree, \eqref{4.1} has a unique positive solution for $\l$ sufficiently close to $\v_0(\l)$, denoted by $(\l,\mu,w_\l,\t_{[\mf{L}_2,\mu,d]})$ in the statement of the theorem. According to \eqref{4.22}, necessarily $(\l,\mu,w_\l,\t_{[\mf{L}_2,\mu,d]})\in\mathscr{C}_0^+$ for $\l\sim \v_0(\mu)$.
\par
The spectrum of the linearization of $\mf{F}_0$ at $(\l_n,\mu,w_n,\t_{[\mf{L}_2,\mu,d]})$ is given by the eigenvalues of the boundary value problem
\begin{equation*}
\left\{
\begin{array}{lll}
\left(\mf{L}_1 +b\dfrac{\t_{[\mf{L}_2,\mu,d]}}{(1+m w_n)^2}-\l_n\right)w=\tau w &\quad \hbox{in}\;\;\Omega,\\[10pt]
\mf{B}_1 w=0 &\quad\hbox{on}\;\;\partial\Omega.
\end{array}
\right.
\end{equation*}
Since $1+m w_n\gneq 1$ for all $n\geq1$, it follows from Theorem \ref{th2.1} and the identity \eqref{4.2} applied to $(\l,w)=(\l_n,w_n)$ that, for every $n\geq 1$,
\begin{equation}
\label{4.28}
\tau_0(n)\equiv  \s_0\left[ \mf{L}_1+b \tfrac{\t_{[\mf{L}_2,\mu,d]}}{(1+mw_n)^2}-\l_n,\mf{B}_1,\O\right]
< \s_0\left[ \mf{L}_1+b \tfrac{\t_{[\mf{L}_2,\mu,d]}}{1+mw_n},\mf{B}_1,\O\right]-\l_n=0.
\end{equation}
On the other hand, it follows from  \eqref{4.27} that
\[
   \lim_{n\to+\infty}\dfrac{b \t_{[\mf{L}_2,\mu,d]}}{(1+mw_n)^2}=
   \left(1-\chi_{_{\mathrm{int\,supp}\,m}}\right) b \t_{[\mf{L}_2,\mu,d]}.
\]
Thus, thanks to \eqref{3.19} and \eqref{4.26}, by letting $n\to \infty$ in \eqref{4.28}, we find that \begin{equation}
\label{4.29}
\lim_{n\to+\infty}\tau_0(n)=\v_0(\mu)-\v_0(\mu)=0,
\end{equation}
though, due to \eqref{4.28}, $\tau_0(n)<0$ for all $n\geq1$.
Similarly, by the strict dominance of the principal eigenvalues, any other eigenvalue, say $\tau_j(n)$, $j\geq 1$, satisfies
\begin{equation*}
\lim_{n\to\infty} \mathrm{Re\,}\tau_j(n)=\mathrm{Re\,}\s_j\left[ \mf{L}_1+\left(1-\chi_{_{\mathrm{int\,supp}\,m}}\right)
 b\t_{[\mf{L}_2,\mu,d]},\mf{B}_1,\O\right]-\v_0(\mu)>0.
\end{equation*}
Therefore, there exists $r>0$ such that any positive solution, $(\l,w)$, of \eqref{4.1}
with $\l\in (\v_0(\mu),\v_0(\mu)+r]$ is non-degenerate with one-dimensional unstable manifold.
This ends the proof.
\end{proof}

Figure \ref{Fig3} shows an admissible component $\mathscr{C}_0^+$ of positive solutions of
\eqref{4.1} respecting Theorems \ref{th4.1} and \ref{th4.2}. Although \eqref{4.1} has a unique positive solution for $\l$ sufficiently close to either
$\Phi(\mu)$, or $\v_0(\mu)$, the problem might possess an arbitrarily large number of positive solutions
for some intermediate range of values of the parameter $\l$, as illustrated in Figure \ref{Fig3}. Actually,
besides $\mathscr{C}_0^+$, \eqref{4.1} might have some additional component of positive solutions not plotted in the figure. In spite of all these circumstances, thanks to Theorems \ref{th4.1} and \ref{th4.2}, near the ends of the existence interval, $(\v_0(\mu),\Phi(\mu))$, the unique positive solution of \eqref{4.1} must be unstable with one-dimensional unstable manifold. It remains an open problem in this paper to analyze the fine structure of the global bifurcation diagram.

\begin{figure}[ht!]
\centering
\includegraphics[scale=0.1]{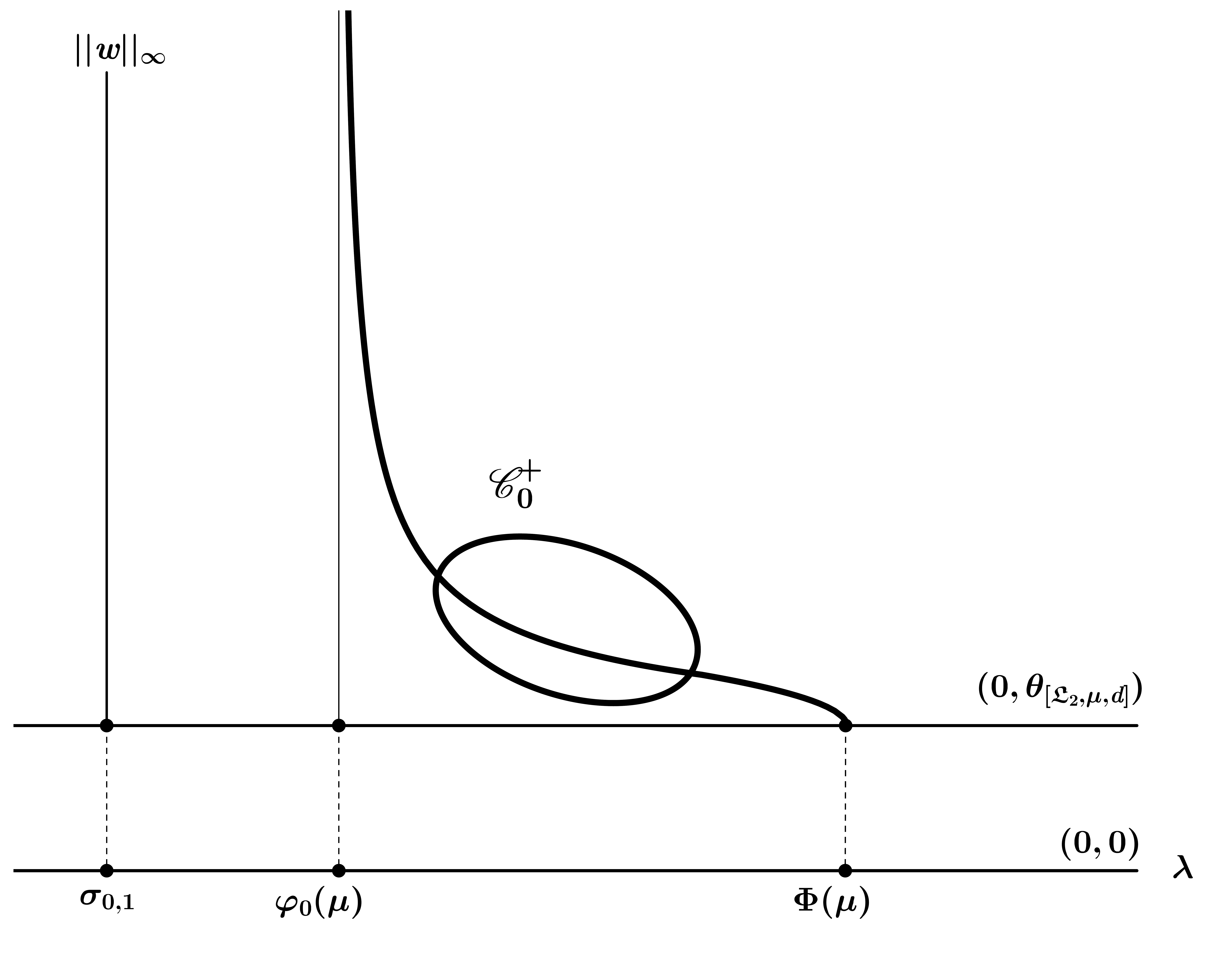}
\caption{An admissible component $\mathscr{C}_0^+$ in case $bm\gneq 0$}
\label{Fig3}
\end{figure}

\section{An optimal multiplicity result for the original model}
\label{sec5}

The next multiplicity result is the main theorem of this section. Remember that, owing to
Theorem \ref{th3.4}, for every $\mu>\s_{0,2}$, \eqref{1.7} has a coexistence state if $\l>\Phi(\mu)$. Moreover, in such case, $\l>\v_\e(\mu)$, because $\Phi(\mu)>\v_\e(\mu)$.

\begin{theorem}
\label{th5.1}
Fix  $\l^*\in(\v_0(\mu),\Phi(\mu))$. Then, there exists $\e_0\equiv \e_0(\l^*)>0$ such that, for every $\e\in(0,\e_0)$, \eqref{1.7} possesses a component $\mathscr{C}_{\e}^+$ of coexistence states satisfying the following properties:
\begin{enumerate}
\item[{\rm (a)}] $\mc{P}_{\l}\left(\mathscr{C}_{\e}^+\right)=[\l_T,+\infty)$ for some $\l_T\equiv\l_T(\e)\in (\v_\e(\mu),\l^*)$.
\item[{\rm (b)}] For every $\l\in [\l^*,\Phi(\mu))$, \eqref{1.7} has, at least, two
(different) coexistence states.
\item[{\rm (c)}] $\mathscr{C}_{\e}^+$ is an analytic $\l$-curve in a neighborhood of
$(\l,\mu,w,v)=(\Phi(\mu),\mu,0,\t_{[\mf{L}_2,\mu,d]})$.
\end{enumerate}
\end{theorem}

Naturally, $\mathscr{C}_\e^+$ is the perturbation of the component $\mathscr{C}_0^+$ constructed in
Section 4 as $\e>0$ leaves $\e=0$. It turns out that, as $\e$ perturbs from zero, the component $\mathscr{C}_0^+$ bends backwards towards the right providing us with a perturbed component like
the one sketched in Figure \ref{Fig5}.

\begin{figure}[ht!]
\centering
\includegraphics[scale=0.1]{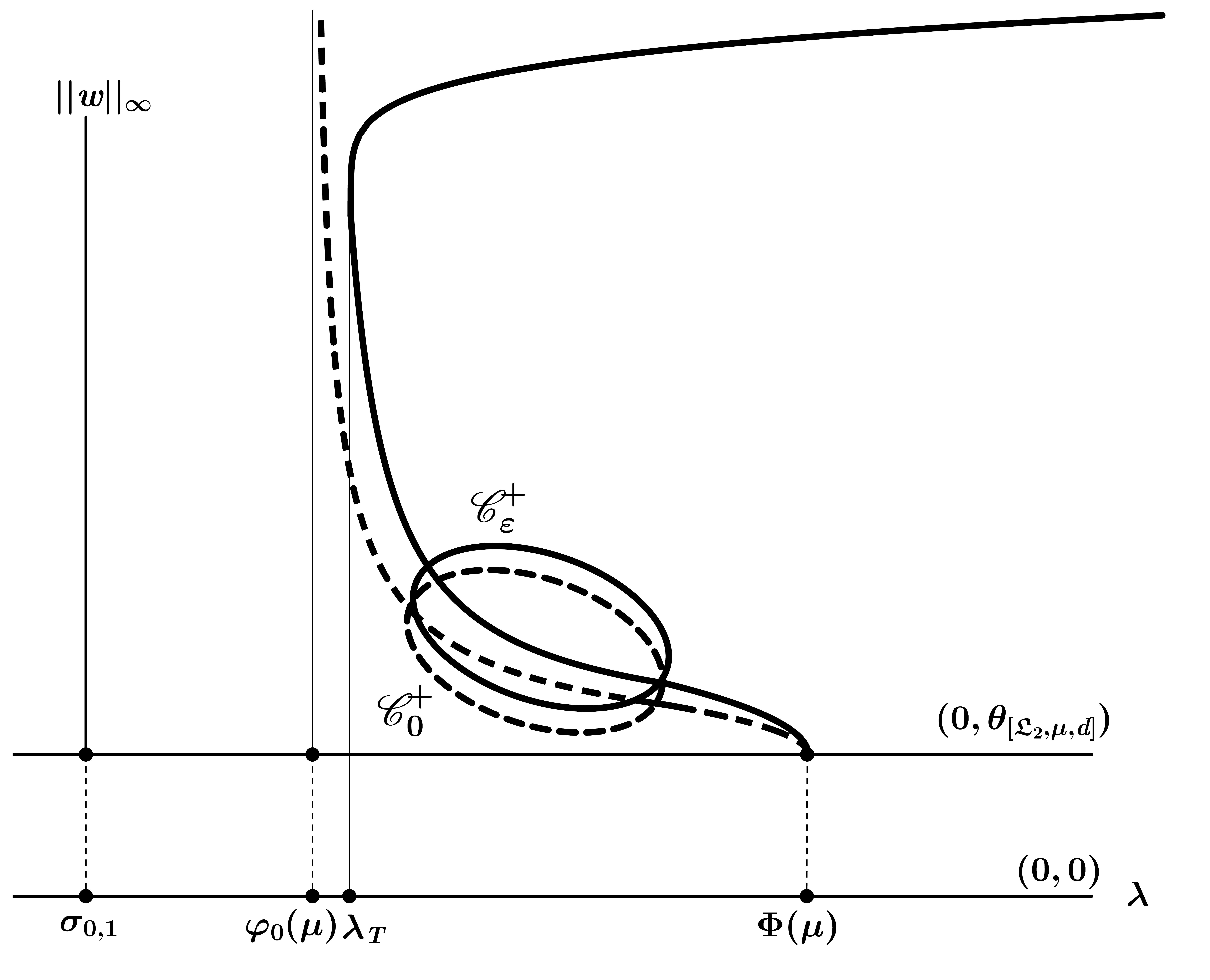}
\caption{The components $\mathscr{C}_0^+$ (dashed line) and $\mathscr{C}_\e^+$ (solid line) for small $\e>0$}
\label{Fig5}
\end{figure}

The proof of Theorem \ref{th5.1} is based on Theorem \ref{th4.2}, Theorem 7.2.2 of  \cite{LG01}, and on the existence of a priori bounds for the coexistence states of \eqref{1.7} established by the following lemma.

\begin{lemma}
\label{le5.1}
Suppose $\e>0$ and let $(w,v)$ be  a coexistence state of \eqref{1.7}. Then,
\begin{equation}
\label{5.1}
0\ll_1w\ll_1\t_{[\mf{L}_1,\l,\e a]},\qquad \t_{[\mf{L}_2,\mu,d]}\ll_2 v\ll_2 \t_{[\mf{L}_2,\mu+\e c\tfrac{\t_{[\mf{L}_1,\l,\e a]}}{1+m\t_{[\mf{L}_1,\l,\e a]}},d]}.
 \end{equation}
\end{lemma}
\begin{proof}
Since $w\gneq0$, by the uniqueness of the principal eigenvalue, it follows from the
$w$-equation of \eqref{1.7} that $w\gg_1 0$ and that
\begin{equation*}
\l=\s_0\left[ \mf{L}_1+\e aw+b \dfrac{v}{1+mw},\mf{B}_1,\O\right].
\end{equation*}
Moreover, it follows from $w\gg_1 0$ that
\begin{equation*}
\mf{L}_1 w=\lambda w - \varepsilon a w^2 -b\dfrac{wv}{1+m w}\lneq \l w-\e aw^2.
\end{equation*}
Thus, $w$ is a positive strict subsolution of the problem
\begin{equation*}
\left\{
\begin{array}{ll}
\mf{L}_1 w=\lambda w - \varepsilon a(x)w^2 &\quad \hbox{in}\;\;\Omega,\\[1ex]
\mf{B}_1 w=0 &\quad\hbox{on}\;\;\partial\Omega.
\end{array}
\right.
\end{equation*}
Hence, since $\l>\s_{0,1}$, it follows from Theorem \ref{th2.3} that
\begin{equation}
\label{5.2}
w\ll_1\t_{[\mf{L}_1,\l,\e a]}.
\end{equation}
This completes the proof of the first two estimates of \eqref{5.1}.  Similarly, by \eqref{5.2},
\begin{align*}
\mu v-d v^2\lneq\mf{L}_2 v&=\mu v -d v^2+ \varepsilon c \dfrac{wv}{1+m w}\\
&\lneq
\left(\mu+ \varepsilon c\dfrac{\t_{[\mf{L}_1,\l,\e a]}}{1+m\t_{[\mf{L}_1,\l,\e a]}}\right)v-d v^2,
\end{align*}
which implies that $v$ is a positive strict supersolution  of
\begin{equation*}
\left\{
\begin{array}{ll}
\mf{L}_2 v=\mu v - d v^2 &\quad \hbox{in}\;\;\Omega,\\[1ex]
\mf{B}_2 v=0 &\quad\hbox{on}\;\;\partial\Omega,
\end{array}	
\right.
\end{equation*}
as well as a positive strict subsolution of
\begin{equation*}
\left\{
\begin{array}{ll}
\mf{L}_2 v=\mu v - d v^2+ \varepsilon c \dfrac{\t_{[\mf{L}_1,\l,\e a]}v}{1+m \t_{[\mf{L}_1,\l,\e a]}} &\quad \hbox{in}\;\;\Omega,\\[1ex]
\mf{B}_2 v=0 &\quad\hbox{on}\;\;\partial\Omega.
\end{array}	
\right.
\end{equation*}
Therefore, the last two estimates of \eqref{5.1} also follow from Theorem \ref{th2.3}.
\end{proof}

The rest of this section is devoted to the proof of Theorem \ref{th5.1}. Throughout it, we fix $\mu>\s_{0,2}$ and  $\l^*\in(\v_0(\mu),\Phi(\mu))$, consider a sufficiently small $r>0$ satisfying the conclusions of
Theorem \ref{th4.2}, and pick $\l_0, \l_1 \in (\v_0(\mu),\Phi(\mu))$ such that
$$
  \l_0 <\v_0(\mu)+r<\l^*<\Phi(\mu)-r< \l_1<\Phi(\mu).
$$
Naturally, $r>0$ can be shortened as much as  necessary. Moreover, for every $t, s \in (\v_0(\mu),\Phi(\mu))$ with $t<s$, we denote by
$\mathscr{C}_{0,[t,s]}^+$ the restriction of the component $\mathscr{C}_0^+$ to the interval $[t,s]$, i.e.,
$$
\mathscr{C}_{0,[t,s]}^+ \equiv \left\{(\l,\mu,w,\t_{[\mf{L}_2,\mu,d]})\in \mathscr{C}_0^+\;:\;\l\in [t,s]\right\}.
$$
By the choice of $\l_0$ and $\l_1$, Theorem \ref{th4.2} guarantees that $\mathscr{C}_{0,[\l_0,\l_1]}^+$ has a unique non-degenerate positive solution for every
\begin{equation}
\label{v.3}
  \l\in [\l_0,\v_0(\mu)+r]\cup [\Phi(\mu)-r,\l_1].
\end{equation}
Actually, by  the implicit function theorem, each of the components $\mathscr{C}_{0,[\l_0,\v_0(\mu)+r]}^+$ and $\mathscr{C}_{0,[\Phi(\mu)-r,\l_1]}^+$ consists of an analytic arc of $\l$-curve. This is a pivotal feature in the proof given here. As these solutions are non-degenerate, once again by the implicit function theorem, these two arcs perturb into two $\l$-arcs of non-degenerate solutions of \eqref{1.7} for sufficiently small $\e>0$.
\par
Now, we consider the bounded set
$$
	\mc{C}_\eta:=\mathscr{C}_{0,[\l_0,\l_1]}^++B_{\eta},
$$
where $B_{\eta}$ stands for the open ball of radius $\eta$ centered at
$(\mu,w,v)=(\mu,0,0)$ in the product space
$$
    \mathscr{X}\equiv \R\times  \mc{C}_{\mf{B}_1}^1(\bar\O)\times  \mc{C}_{\mf{B}_2}^1(\bar\O).
$$
Then, $\mc{C}_\eta$ is a $\eta$-neighborhood of $\mathscr{C}_{0,[\l_0,\l_1]}^+$ with side covers
$$
  \{\l_0\}\times [(\mu,w_{\l_0},\t_{[\mf{L}_2,\mu,d]})+ B_\eta],\qquad
  \{\l_1\}\times [(\mu,w_{\l_1},\t_{[\mf{L}_2,\mu,d]})+ B_\eta],
$$
where $w_\l$ denotes the unique positive solution of \eqref{4.1} for every $\l$ satisfying \eqref{v.3}. By construction, $\mathscr{C}_{0,[\l_0,\l_1]}^+\subset \mc{C}_\eta$. Moreover, for sufficiently small $\eta>0$,
$$
  (\l,\mu,w,v)=(\l,\mu,w_\l,\t_{[\mf{L}_2,\mu,d]})
$$
is the unique solution of \eqref{1.8} in $\mc{C}_\eta$ for each $\l$ satisfying \eqref{v.3}. Furthermore, since the $w$-components of the elements of $\mathscr{C}_{0,[\l_0,\l_1]}^+$ are separated away from zero, because $\l=\Phi(\mu)$ is the unique bifurcation value to coexistence states from $w=0$, $\bar{\mc{C}}_\eta$ cannot admit any solution of the form $(\l,\mu,0,v)$ with $v=0$ or $v=\t_{[\mf{L}_2,\mu,d]}$ for sufficiently small $\eta>0$.
\par
Next, we will adapt the proof of \cite[Th. 6.3.1]{LG01}, through a well known lemma of Whyburn \cite[Ch. 1]{Why-1964} on compact continua, to show that, if necessary,  $\mc{C}_\eta$ can be shortened in the interval
$[\v_0(\mu)+r,\Phi(\mu)-r]$ up to obtain an \emph{isolating neighborhood} of $\mathscr{C}_0^+$, denoted by $\mc{O}$, in the sense that, besides the previous properties of $\mc{C}_\eta$,  $\p_L\mc{O}\cap \mathscr{S}_0$ cannot admit any positive solution of \eqref{1.8}, $(\l,\mu,w,\t_{[\mf{L}_2,\mu,d]})$,  with
$\l \in [\v_0(\mu)+r,\Phi(\mu)-r]$. We are denoting by  $\p_L \mc{O}$ the set $\p\mc{O}$, except for the two
lateral side covers at $\l=\l_0$ and $\l=\l_1$, where $\mathscr{S}_0$ has exactly two non-degenerate coexistence states. Indeed, should $\mc{C}_\eta$ satisfy this property we can take
$\mc{O}=\mc{C}_\eta$. Otherwise, we consider the non-empty compact sets
\begin{align*}
M & := \left\{(\l,\mu,w,\t_{[\mf{L}_2,\mu,d]})\in \bar{\mc{C}}_\eta\cap\mathscr{S}_0\;:\;
\l \in [\v_0(\mu)+r,\Phi(\mu)-r]\right\},\\
A & :=\left\{(\l,\mu,w,\t_{[\mf{L}_2,\mu,d]})\in \partial\mc{C}_\eta\cap\mathscr{S}_0\;:\;
\l \in [\v_0(\mu)+r,\Phi(\mu)-r]\right\},\\
B & :=\mathscr{C}_{0,[\v_0(\mu)+r,\Phi(\mu)-r]}^+.
\end{align*}
These sets are compact because they are closed and bounded sets consisting of fixed points of a compact operator. Moreover, $A$ and $B$ are disjoint. Thus, according to  Whyburn \cite[Ch.1]{Why-1964}, since $B$ is a connected component, there are two disjoint compact subsets of $M$, $M_A$ and $M_B$, such that $A\subset M_A$,
$B\subset M_B$ and $M=M_A\cup M_B$. Thus, setting $\delta:={\rm dist}(M_A,M_B)>0$, it is easily seen that
\[
\mc{O}:=\mc{C}_\eta\setminus \overbar{M_A+B_{\frac{\d}{2}}}
\]
satisfies similar properties as $\mc{C}_\eta$ and, in addition, by construction,
\begin{equation}
\label{v.4}
   \p_L\mc{O}\cap \mathscr{S}_0=\emptyset,
\end{equation}
 This construction has been sketched in  Figure \ref{Fig4}, where an admissible $\mc{O}$ has been plotted  when $\mc{O}=\mc{C}_\eta$.

\begin{figure}[ht!]
	\centering
	\includegraphics[scale=0.40]{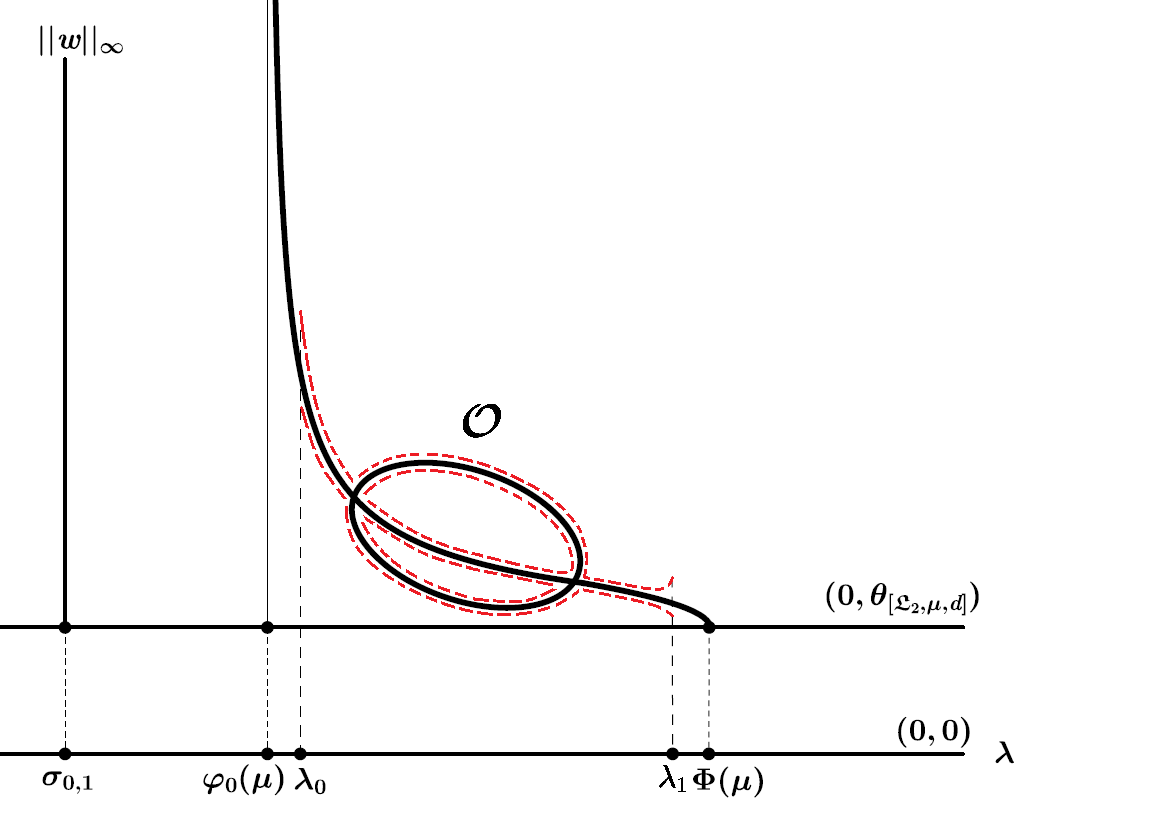}
	\caption{The isolating neighborhood $\mc{O}$ of $\mathscr{C}_{0,[\l_0,\l_1]}^+$}
	\label{Fig4}
\end{figure}

Subsequently, for sufficiently small $\e>0$, we denote by $\mathscr{S}_\e$ the set of nontrivial solutions of \eqref{1.7},
$$
   \mathscr{S}_\e:=\{(\l,\mu,w,v)\in\mf{F}^{-1}(0)\,:\,(w,v)\neq (0,\t_{[\mf{L}_2,\mu,d]})\}\cup\{(\l,\mu,0,\t_{[\mf{L}_2,\mu,d]})\;:\;\l\in\Sigma(\mathscr{L}(\l,\e))\},
$$
where $\Sigma(\mathscr{L}(\l,\e))$ is the generalized spectrum of $\mathscr{L}(\l,\e)$, as discussed in
\cite{LG01}. By \cite[Th. 7.2.2]{LG01}, there exists a component of $\mathscr{S}_\e$, denoted by
$\mathscr{C}_\e^+$, consisting of coexistence states of \eqref{1.7} such that
$$
  (\Phi(\mu),\mu,0,\t_{[\mf{L}_2,\mu,d]}) \in \mathscr{\bar C}_\e^+.
$$
 However, contrarily to what happens with $\mathscr{C}_0^+$, Lemma \ref{le5.1} entails that, for every $\e>0$ and $\hat\l>\Phi(\mu)$, the set of coexistence states
\begin{equation*}
	\left\{(\l,\mu,w,v)\in \mathscr{C}_\e^+\;:\; \l \in (\v_\e(\mu),\hat\l]\right\}
\end{equation*}
is bounded, whereas, thanks to \cite[Th. 7.2.2]{LG01}, $\mathscr{C}_\e^+$ is unbounded. Consequently,
as soon as $\l'(0)<0$, which holds true for sufficiently small $\e>0$, there exists $\l_T\equiv \l_T(\e)\in (\v_\e(\mu),\Phi(\mu))$ such that
\begin{equation*}
\mc{P}_\l \left( \mathscr{C}_\e^+\right)= [\l_T(\e),+\infty).
\end{equation*}
We claim that $\l_T(\e)<\l_0$ for sufficiently small $\e>0$. Since $\l_0<\l^*$, this ends the
proof of Part (a). Note that $\l_T>\v_\e(\mu)$ by Theorem \ref{th3.4}. To prove $\l_T(\e)<\l_0$, we first
show that
\begin{equation}
\label{v.5}
[\l_0,\Phi(\mu))\subset \mc{P}_\l \left( \mathscr{C}_\e^+\right)\quad \hbox{for sufficiently small}\;\;
\e>0.
\end{equation}
This holds true thanks to the crucial feature that the isolating neighborhood of $\mathscr{C}_0^+$ in $[\l_0,\l_1]$, $\mc{O}$, also provides us with an isolating neighborhood of $ \mathscr{C}_\e^+$ in $[\l_0,\l_1]$ for sufficiently small $\e>0$ if $\l_1$ is sufficiently close to $\Phi(\mu)$. Indeed, thanks to Theorem \ref{th3.2}, one can choose $\l_1$ to be sufficiently close to $\Phi(\mu)$ so that, for sufficiently small $\e>0$, $ \mathscr{C}_\e^+$  has a unique non-degenerate coexistence state close to $(w,v)=(0,\t_{[\mf{L}_2,\mu,d]})$ for all $\l \in [\l_1,\Phi(\mu))$, say
$$
   (\l,\mu,w,v)=(\l,\mu,w_{\l,\e},v_{\l,\e}),\qquad \l \in [\l_1,\Phi(\mu)),\;\; \e \in [0,\e_0).
$$
Naturally, as Theorem \ref{th3.2} shows that $ \mathscr{C}_\e^+$ is  a regular perturbation of $\mathscr{C}_0^+$ through the implicit function theorem in a neighborhood of the bifurcation point, there exists $\e_0>0$ such that, for every $\e\in [0,\e_0)$, the coexistence state $(\l_1,\mu,w_{\l_1,\e},v_{\l_1,\e})$ lies in the interior of the right side cover of $\mc{O}$; actually, it is the unique coexistence state of \eqref{1.7} on $\p\mc{O}$ for $\l=\l_1$. This argument combined with the local uniqueness of Theorem \ref{th3.2} shows Part (c). Figure \ref{Fig6} sketches this behavior. As in the remaining bifurcation diagramas plotted in this section, the dashed curve represents $\mathscr{C}_0^+$, while
the continuous curve shows $\mathscr{C}_\e^+$ for sufficiently small $\e>0$.
According to Theorem \ref{th3.2}, these are the unique solutions of the model in a neighborhood of the bifurcation point for sufficiently
small $\e\geq 0$. All are non-degenerate; actually, linearly unstable with one-dimensional unstable
manifold by the exchange stability principle.

\begin{figure}[ht!]
\centering
\includegraphics[scale=0.4]{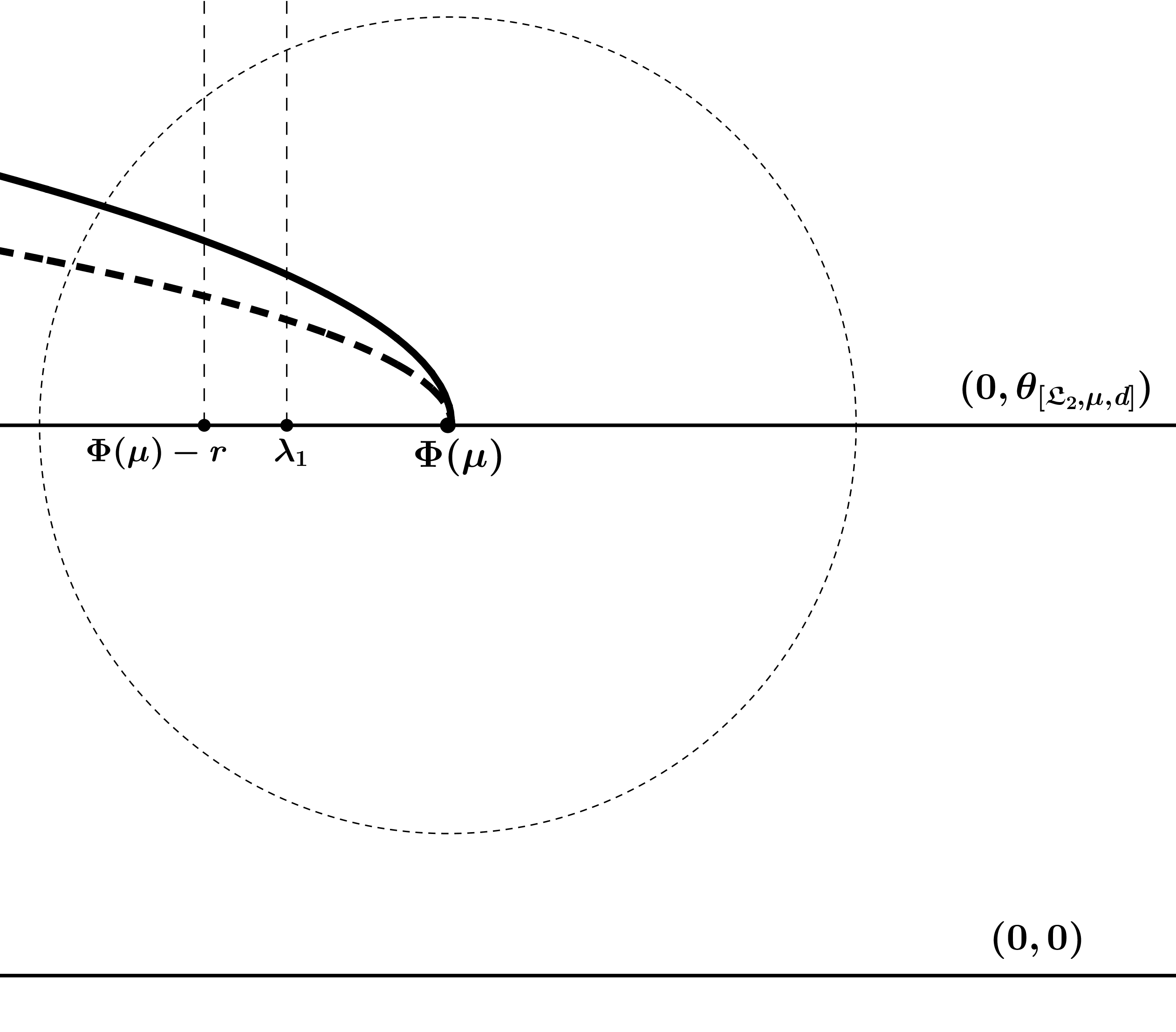}
\caption{The ball where the solutions of \eqref{1.7} are analytic $\l$-curves}
\label{Fig6}
\end{figure}

Once shown that $\mathscr{C}_\e^+$ reaches $\mc{O}$ at $\l=\l_1$, and so enters into $\mc{O}$, we claim that these components  must abandone $\mc{O}$ passing through some point with $\l=\l_0$, as illustrated by the
left picture of Figure \ref{Fig7}, so concluding the proof of \eqref{v.5}. Since they must abandone $\mc{O}$ because they are unbounded, in order to prove our claim, it suffices to make sure that $\mathscr{C}_\e^+$ cannot leave $\mc{O}$ through $\p_L\mc{O}$ for sufficiently small $\e>0$, as illustrated by the right picture of
Figure \ref{Fig7}.  Our proof of this fact proceeds by contradiction. Assume that there is a sequence $\e_n$, $n\geq 2$, such that $\lim_{n\uparrow \infty}\e_n=0$, and, for every $n\geq 2$, $\l'(\e_n)<0$ and
the problem \eqref{1.7} has some coexistence state, $(\l_n,\mu,w_n,v_n)\in \p_L\mc{O}$, for $\e=\e_n$ and some $\l_n\in [\v_0(\mu)+r,\Phi(\mu)-r]$, as sketched on the right picture of Figure \ref{Fig7}.

\begin{figure}[ht!]
	\centering
	\includegraphics[scale=0.09]{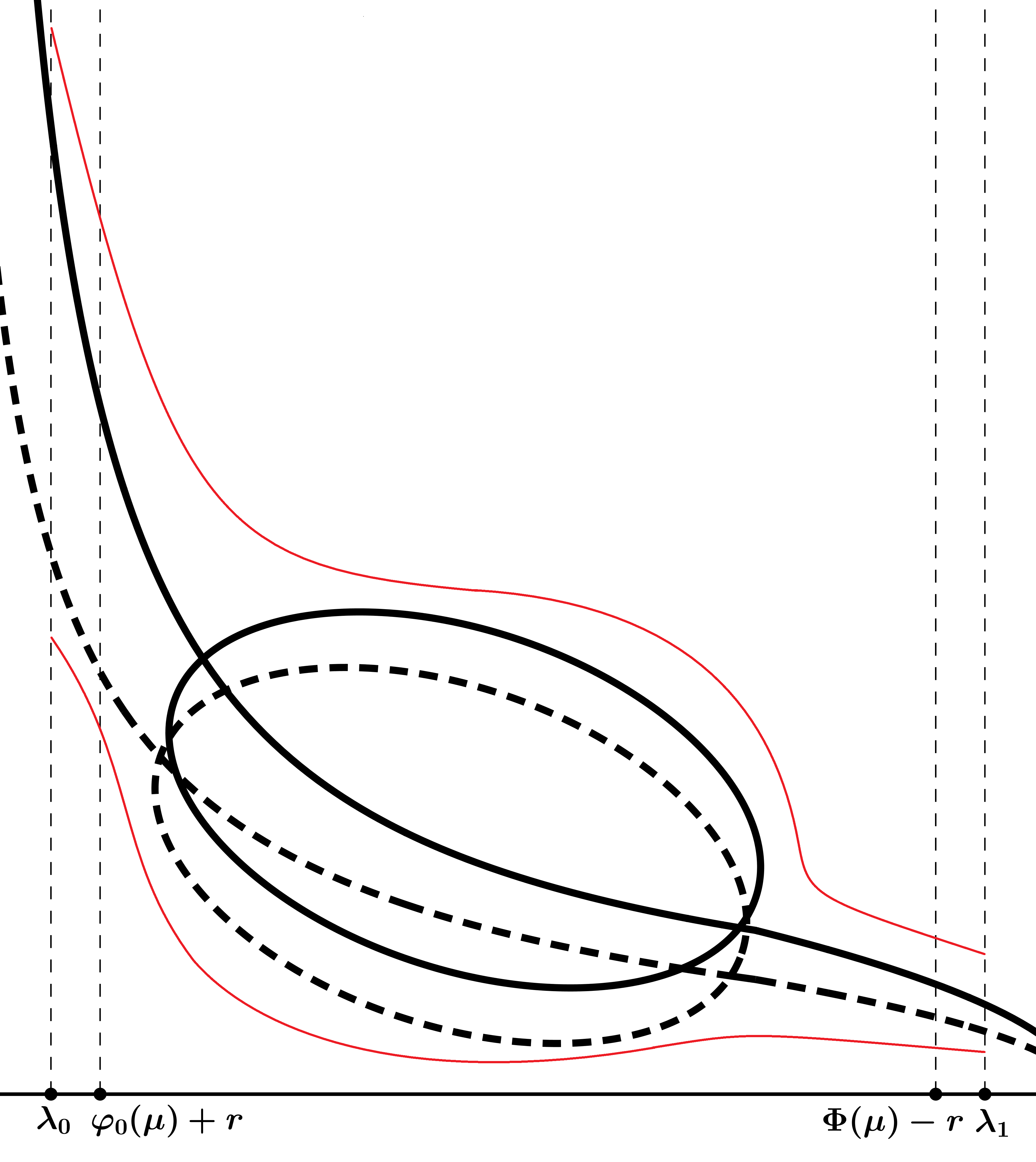}\qquad\quad
	\includegraphics[scale=0.079]{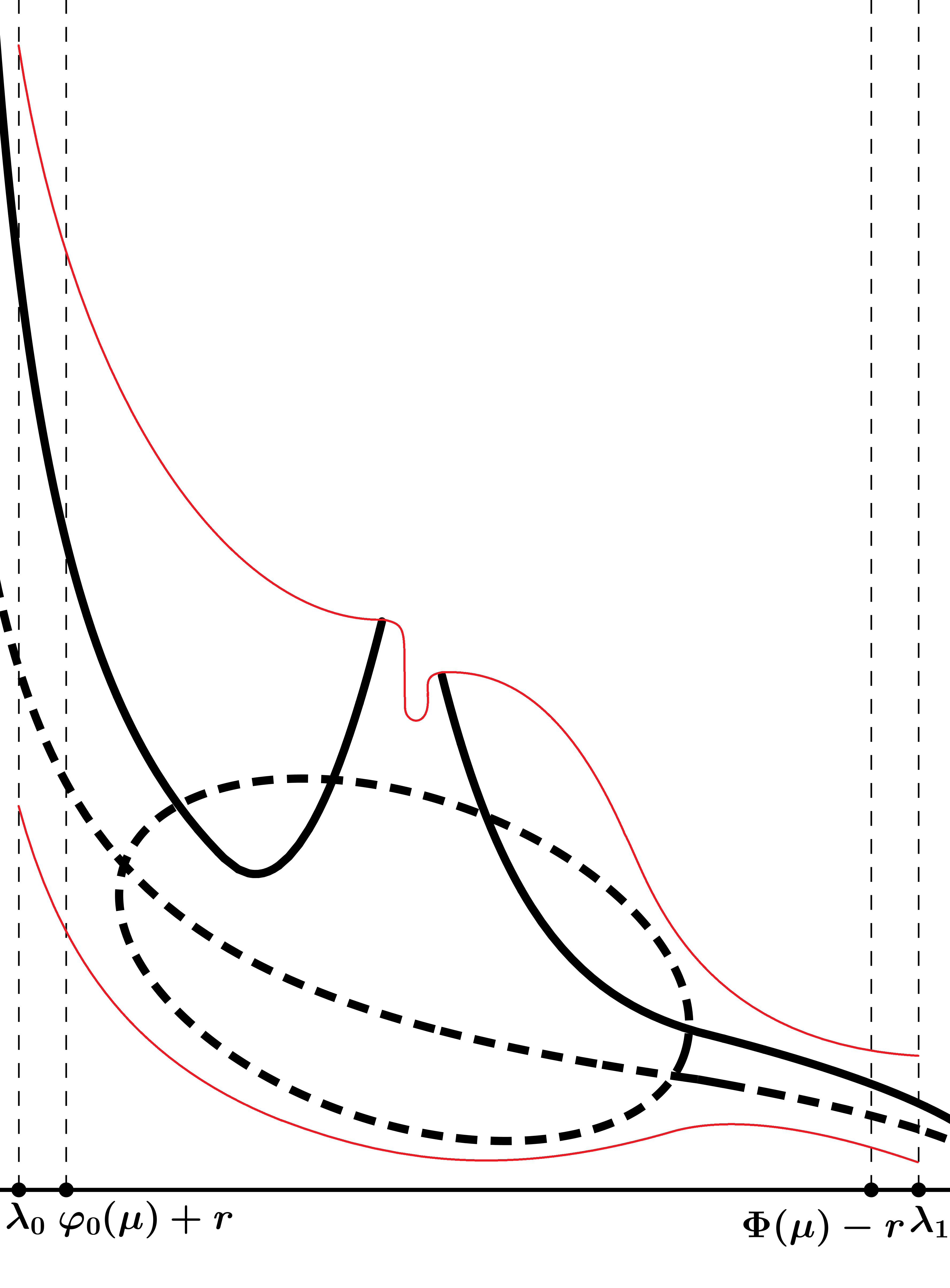}
	\caption{The isolating neighborhood $\mc{O}$ of $\mathscr{C}_{0,[\l_0,\l_1]}^+$}
	\label{Fig7}
\end{figure}

Then, since  $\{(\l_n,\mu,w_n,v_n)\}_{n\geq 2}$ is bounded in $[\l_0,\l_1]\times\{\mu\}\times \mc{C}_{\mf{B}_1}^1(\bar\O)\times \mc{C}_{\mf{B}_2}^1(\bar\O)$ and it consists
of fixed points of a sequence of associated compact operators depending continuously on $\e$, $\e\sim 0$, by a rather standard compactness argument, we can extract a subsequence, relabeled by $n\geq 2$, such that
$$
\lim_{n\to \infty}(\l_n,\mu,w_n,v_n)=(\l_\o,\mu,w_\o,v_\o)\in\partial_L \mc{O}
$$
for some $w_\o\geq 0$, $v_\o\geq0$ and $\l_\o\in[\l_0,\l_1]$ such that $(\l_\o,\mu,w_\o,v_\o)$ solves \eqref{1.8}. Since  $\mc{O}$ is an isolating neighborhood of $\mathscr{C}_0^+$, it becomes apparent that $w_\o\gg_1 0$ and $v_\o\gg_2 0$. But this contradicts \eqref{v.4}. Therefore, \eqref{v.5} holds true. Consequently, for every $\e\in [0,\e_0)$, we have that $\l_T(\e)\leq \l_0<\l^*$, which ends the proof of Part (a).
\par
Note that, as $\e>0$ perturbs from zero, a further application of the implicit function theorem shows that the analytic arcs of $\l$-curve $\mathscr{C}_{0,[\l_0,\v_0(\mu)+r]}^+$ and $\mathscr{C}_{0,[\Phi(\mu)-r,\l_1]}^+$ perturb into two $\l$-arcs of $\mathscr{C}_\e^+$ within $\mc{O}$, denoted by $\mathscr{C}_{\e,[\l_0,\v_0(\mu)+r]}^+$ and $\mathscr{C}_{\e,[\Phi(\mu)-r,\l_1]}^+$,  and that these arcs consist of non-degenerate solutions of \eqref{1.7} for sufficiently small $\e\geq 0$. By a further application of the implicit function theorem at the unique solution of $\mathscr{C}_\e^+$
on $\p \mc{O}$ at $\l_0$, say $(\l_0,\mu,w_{\l_0,\e},v_{\l_0,\e})$,  this entails that actually for sufficiently small $\e>0$ there exists $\d(\e)>0$ such that
\begin{equation*}
[\l_0-\d(\e),\Phi(\mu))\subset \mc{P}_\l \left( \mathscr{C}_\e^+\right).
\end{equation*}
Moreover,
$$
  \lim_{\e\da 0} (w_{\l_0,\e},v_{\l_0,\e})=(w_{\l_0},\t_{[\mf{L}_2,\mu,d]}).
$$
Therefore, since
$$
   \mathscr{C}_\e^+\setminus  \mathscr{C}_{\e,[\l_0,\v_0(\mu)+r]}^+
$$
is unbounded, it follows from Lemma \ref{le5.1} that, for every $\l \in [\l_0,\Phi(\mu))$, \eqref{1.7}
has, at least, two coexistence states for sufficiently small $\e>0$. This proves Part (b) and concludes the
proof of Theorem \ref{th5.1}.
\par
Another proof of the multiplicity result of Part (b) can be given by using the topological degree. Although this proof does not allow to show that each of the components $\mathscr{C}_\e^+$ bend backwards at some
supercritical turning point for sufficiently small $\e>0$, it provides with the local index of the additional solutions, which is imperative to ascertain their local stability character. The alternative proof proceeds as follows. Thanks to Theorem \ref{th4.2}, it follows from the invariance by homotopy of the Leray--Schauder degree, that, for every $\e\in [0,\e_0)$ and $\l\in [\l_0,\l_1]$,
\begin{equation}
\label{v.6}
\mathrm{Deg\,} (\mf{F}( \l,\mu,\e,\cdot,\cdot),\mc{O}_\l)=-1,
\end{equation}
where $\mf{F}$ is the operator defined in \eqref{3.8} and, for every $\l\in [\l_0,\l_1]$, we are denoting
$$
  \mc{O}_\l :=\{(\mu,w,v)\in \mathscr{X}\;:\; (\l,\mu,w,v)\in \mc{O}\}.
$$
Subsequently we will use the fixed point index in cones, as axiomatized by Amann \cite{Am-1976} and Dancer \cite{Da-1983}, which was applied by the first time to the classical diffusive Lotka--Volterra models by L\'{o}pez-G\'{o}mez and Pardo \cite{LGP-1992}, L\'{o}pez-G\'{o}mez \cite{LG-1992} and, more recently, by Fern\'{a}ndez-Rinc\'{o}n and L\'{o}pez-G\'{o}mez \cite{FRLG-2021}, among many others. First, we consider, for every $i=1,2$, the positive cone of $\mathscr{W}_i$,
\[
   \mathscr{P}_{\mathscr{W}_i}:=\{u \in\mathscr{W}_i \; :\; u\geq0\;\;\hbox{in}\;\O\}
\]
and the associated system to \eqref{1.7}
\begin{equation}
\label{v.7}
\left\{
\begin{array}{lll}
\mf{L}_1 w=\lambda w - \varepsilon a(x)w^2 -\a b(x)\dfrac{wv}{1+m(x)w} &\quad \hbox{in}\;\;\Omega,\\[10pt]
\mf{L}_2 v=\mu v -d(x)v^2+ \a\varepsilon c(x)\dfrac{wv}{1+m(x)w}&\quad \hbox{in}\;\;\Omega,\\[10pt]
\mf{B}_1 w=\mf{B}_2 v=0 &\quad\hbox{on}\;\;\partial\Omega,
\end{array}
\right.
\end{equation}
where $\a\in[0,1]$ is an homotopy parameter to uncouple \eqref{1.7} into two semilinear boundary value problems.  By applying Lemma \ref{le5.1} uniformly in $\a\in [0,1]$, it is easily seen that there exists a bounded open subset $\mc{W}\times\mc{V}\subset\mathscr{W}_1\times\mathscr{W}_2$, independent of $\a\in[0,1]$,
such that $(w,v) \in \mc{W}\times\mc{V}$ if $(w,v)\in P_{\mathscr{W}_1}\times P_{\mathscr{W}_2}$ solves
\eqref{v.7} for some $\a\in[0,1]$.
\par
Subsequently, we choose a sufficiently large $e\geq0$ such that
\begin{equation}
\label{v.8}
\s_0[\mf{L}_i+e,\mf{B}_i,\O]>1,\qquad i=1, 2,
\end{equation}
and, for every $\a\in[0,1]$, $w\in\mc{W}$, and $v\in\mc{V}$,
\begin{equation}
\label{v.9}
\l-a\e w-\a b\frac{v}{1+mw}+e>0, \qquad
\mu-dv+\a \e c\frac{w}{1+mw}+e>0
\qquad\hbox{in}\;\,\bar{\O}.
\end{equation}
Then, thanks to \eqref{v.8} and \eqref{v.9}, the map
\[
\mc{H}:[0,1]\times\mc{W}\times\mc{V}\rightarrow \mathscr{W}_1\times\mathscr{W}_2
\]
defined by
\[
\mc{H}(\a,w,v)=
\left(
\begin{array}{ll}
&(\mf{L}_1+e)^{-1}[(\l-\e aw-\a b\frac{v}{1+mw}+e)w]\\[5pt]
&(\mf{L}_2+e)^{-1}[(\mu-dv+\a\e c\frac{w}{1+mw}+e)v]
\end{array}
\right),
\]
is a compact order preserving operator whose non-negative fixed points are the solutions of \eqref{v.7} in $P_{\mathscr{W}_1}\times P_{\mathscr{W}_2}$. Adapting the analysis of  Steps i)-v) of the proof of \cite[Th. 4.1]{LG-1992}, or Lemmas 5.6-5.9 of \cite{FRLG-2021}, one can find out the fixed point indices of the non-negative solutions of \eqref{1.7} as fixed points of $\mc{H}(1,\cdot,\cdot)$. It turns out that
\begin{equation}
\label{v.10}
i_{P_{\mathscr{W}_1}\times P_{\mathscr{W}_2}}\left(\mc{H}(1,\cdot,\cdot),\mc{W}\times\mc{V}\right)=1,
\end{equation}
whereas
\begin{equation}
\label{v.11}
i_{P_{\mathscr{W}_1}\times P_{\mathscr{W}_2}}\left(\mc{H}(1,\cdot,\cdot),(0,0)\right)=0
\quad\hbox{if}\;\,\l>\s_{0,1}\;\;\hbox{or}\;\; \mu>\s_{0,2}.
\end{equation}
Moreover,
\begin{equation}
\label{v.12}
\left\{
\begin{array}{ll}
i_{P_{\mathscr{W}_1}\times P_{\mathscr{W}_2}}\left(\mc{H}(1,\cdot,\cdot),(\t_{[\mf{L}_1,\l,a]},0)\right)=0&\qquad\hbox{if}\;\;\mu>\Psi_\e(\l),\\[5pt]
i_{P_{\mathscr{W}_1}\times P_{\mathscr{W}_2}}\left(\mc{H}(1,\cdot,\cdot),(0,\t_{[\mf{L}_2,\mu,d]})\right)=1&\qquad\hbox{if}\;\;\l<\Phi(\mu).
\end{array}
\right.
\end{equation}
Thus, for every $\l\in (\s_{0,1},\Phi(\mu))$ and $\mu>\s_{0,2}$,
\begin{align*}
1 & = i_{P_{\mathscr{W}_1}\times P_{\mathscr{W}_2}}\left(\mc{H}(1,\cdot,\cdot),\mc{W}\times\mc{V}\right) =  i_{P_{\mathscr{W}_1}\times P_{\mathscr{W}_2}}\left(\mc{H}(1,\cdot,\cdot),(0,0)\right) \\ & \hspace{1.1cm} + i_{P_{\mathscr{W}_1}\times P_{\mathscr{W}_2}}\left(\mc{H}(1,\cdot,\cdot),(\t_{[\mf{L}_1,\l,a]},0)\right)+i_{P_{\mathscr{W}_1}\times P_{\mathscr{W}_2}}\left(\mc{H}(1,\cdot,\cdot),(0,\t_{[\mf{L}_2,\mu,d]})\right).
\end{align*}
Consequently, the global index of the coexistence states, as fixed points of
 $\mc{H}(1,\cdot,\cdot)$, equals zero and, since \eqref{v.6} entails
$$
  i_{P_{\mathscr{W}_1}\times P_{\mathscr{W}_2}}\left(\mc{H}(1,\cdot,\cdot),\mc{O}_\l\right)=-1\quad
  \hbox{for all}\;\; \l\in[\l_0,\l_1],
$$
the existence of  a second coexistence state follows for every
$\l\in [\l_0,\l_1]$. Taking into account that $\l_0<\l^*$ and that
$\l_1$ can be chosen arbitrarily close to $\Phi(\mu)$, the multiplicity result of
Theorem \ref{th5.1}(b) readily follows.
\par

\begin{remark}
\label{re5.1} \rm
The multiplicity result of Theorem \ref{th5.1}(b) holds as soon as $\l'(\e)<0$, which
occurs for $\e\in [0,\e^*)$, where $\l'(\e^*)=0$. It remains an open problem to ascertain whether, or not, \eqref{1.7} can admit a coexistence state for some $\l \in (\v_\e(\mu),\l_T(\e))$. This might depend on the nature of the spatial heterogeneities of the several coefficients involved in the setting of \eqref{1.7}.
\end{remark}

\section{A simple illustrative example}
\label{sec6}

This section considers \eqref{1.7} in the special case when:
\begin{itemize}
\item $c_1=c_2=0$ in $\O$.
\item $\G_1=\p\O$ (i.e., $\G_0=\emptyset$), and $\b_1=\b_2=0$ on $\p\O$.
\item $a$, $b$, $c$ and $d$ are positive constants, and $m = 1$ in $\O$.
\end{itemize}
Then, since $\mf{B}_\kappa = \frac{\p}{\p \nu_\kappa}\equiv \partial_{\nu_\k}$ for $\kappa = 1, 2$,
it turns out that we are imposing non-flux boundary conditions on $\p\O$. Thus,
$$
   \s_{0,\kappa} := \s_0[\mf{L}_\k,\partial_{\nu_\k},\Omega]=0,\qquad \kappa=1,2.
$$
Consequently, throughout this section we assume that $\l>0$ and fix $\mu>0$. As in the preceding sections, $\l>0$ is regarded as a bifurcation parameter.
\par
By the special nature of \eqref{1.7} under these conditions, any component-wise positive solution $(w,v)$ of the algebraic system
\begin{equation}
\label{6.1}
\left\{
\begin{array}{ll}
\lambda-\varepsilon aw-bv\frac{1}{1+w}=0,\\[1ex]
\mu-dv+\varepsilon c\frac{w}{1+w}=0,
\end{array}
\right.
\end{equation}
provides us with a coexistence state of \eqref{1.7}. By the uniqueness
of Theorem \ref{th2.3}, it follows that $\t_{[\mf{L}_2,\mu,d]}=\frac{\mu}{d}$. So,
\begin{equation}
\label{vi.2}
   \Phi(\mu)=\s_0\left[ \mf{L}_1+b\t_{[\mf{L}_2,\mu,d]},\partial_{\nu_1},\O\right]=b\frac{\mu}{d}.
\end{equation}
Eliminating $v$ from the first equation of \eqref{6.1}, we obtain that
\begin{equation}
\label{vi.3}
   v=\frac{1}{b}(1+w)(\lambda-\varepsilon aw),
\end{equation}
and, substituting \eqref{vi.3} into the second equation of \eqref{6.1}, yields to
\begin{equation}
\label{vi.4}
P(w,\l)\equiv P(w):=w^3+\left(2-\frac{\lambda}{\varepsilon a}\right)w^2+\left(1+\frac{bc}{ad}+\frac{b\mu-2d\lambda}{\varepsilon ad}\right)w+\frac{b \mu -d \lambda}{\varepsilon ad}=0.
\end{equation}
Therefore,  $(w,v)$ is component-wise positive solution of the system \eqref{6.1} if, and only if,  $w$ is a positive root of $P(w)$ with
\begin{equation}
\label{vi.5}
  \lambda-\varepsilon aw>0.
\end{equation}
Thus, to find out the coexistence states of \eqref{1.7} for this prototype, one should first ascertain the positive roots of $P(w)$. In this section, we are going to accomplish this task for $\l>\Phi(\mu)$ sufficiently close to $\Phi(\mu)$. Note that, according to the analysis of the previous sections, we already know that
$(\l,w,v)=\left(\Phi(\mu),0,\t_{[\mf{L}_2,\mu,d]}\right)$
is a bifurcation point to a component of coexistence states of \eqref{1.7}.
\par
Suppose $\l >\Phi(\mu)$. Then, by \eqref{vi.2}, $\l >b\frac{\mu}{d}$. Thus,
$$
  P(0)= \frac{b \mu -d \lambda}{\varepsilon ad}<0,
$$
and hence, since $\lim_{w\uparrow +\infty}P(w)=+\infty$, $P(w)$ admits, at least, a positive real root.
Similarly, the polynomial
$$
  P'(w)= 3 w^2+2\left(2-\frac{\lambda}{\varepsilon a}\right)w+ 1+\frac{bc}{ad}+\frac{b\mu-2d\lambda}{\varepsilon ad}
$$
satisfies
$$
  P'(0)= 1+\frac{bc}{ad}+\frac{b\mu-2d\lambda}{\varepsilon ad} <0
$$
if, and only if,
\begin{equation}
\label{vi.6}
  0<\e <\e^*(\l)\equiv \frac{2d\lambda-b\mu}{bc+ad}.
\end{equation}
So, since $\lim_{w\uparrow +\infty} P'(w)= +\infty$, also $P'(w)$ possesses, at least, one
positive root for every $\e \in (0,\e^*)$. Finally, since
$$
  P''(w)= 6w +2\left(2-\frac{\lambda}{\varepsilon a}\right),
$$
it is obvious that $w_c \equiv \frac{1}{3}\left( \frac{\lambda}{\varepsilon a}-2\right)$
is the unique root of $P''$. Suppose that
\begin{equation}
\label{vi.7}
   \e < \min\left\{ \frac{\lambda}{2a},\e^*(\l)\right\}.
\end{equation}
Then $w_c>0$, $P''(w)<0$ if $w\in [0,w_c)$, and $P''(w)>0$ if $w>w_c$. Thus, $P'(w)$ is decreasing
in $(0,w_c)$ and increasing in $(w_c,+\infty)$.  Moreover, by \eqref{vi.6}, $P'(0)<0$, because $\e<\e^*(\l)$. Consequently, there exists $w^*>0$ such that $P'<0$ in $[0,w^*)$, $P'(w^*)=0$, and $P'(w)>0$ for all $w>w^*$. Therefore, $P(w)$ is decreasing in $(0,w^*)$ and increasing in $(w^*,+\infty)$, and, since $P(0)<0$ and $P'(0)<0$, it becomes apparent that, under condition \eqref{vi.7}, $P(w)$ has a unique positive root, say $w_r>w^*>0$. Finally, since
\begin{align*}
  P\left(\frac{\lambda}{\varepsilon a}\right) & =\left(\frac{\lambda}{\varepsilon a}\right)^3 +
  \left(2-\frac{\lambda}{\varepsilon a}\right)\left(\frac{\lambda}{\varepsilon a}\right)^2 +
  \left(1+\frac{bc}{ad}+\frac{b\mu-2d\lambda}{\varepsilon ad}\right) \frac{\lambda}{\varepsilon a}+ \frac{b \mu -d \lambda}{\varepsilon ad}\\ &   =2 \left(\frac{\lambda}{\varepsilon a}\right)^2 +
  \left(1+\frac{bc}{ad}+\frac{b\mu-2d\lambda}{\varepsilon ad}\right) \frac{\lambda}{\varepsilon a}+ \frac{b \mu -d \lambda}{\varepsilon ad}\\ & =
  \frac{1}{\e^2}\Big[ 2\frac{\lambda^2}{a^2}+ \frac{b\mu-2\lambda d}{a^2d}\l+O(\varepsilon)\Big]=
  \e^{-2}\Big[\frac{\l}{a^2} \Phi(\mu)+O(\varepsilon)\Big]> 0
\end{align*}
as $\e\downarrow 0$, necessarily  $w_r<\frac{\lambda}{\varepsilon a}$ for sufficiently small $\varepsilon>0$ and, in particular, $w=w_r$ satisfies \eqref{vi.5}. Therefore, for sufficiently small
$\varepsilon>0$, \eqref{6.1} has a unique coexistence state for every $\lambda>\Phi(\mu)$.
\par
Note that, at $\e=0$, \eqref{6.1} reduces to
$$
\left\{
\begin{array}{ll}
\lambda-bv\frac{1}{1+w}=0,\\[1ex]
\mu-dv=0,
\end{array}
\right.
$$
whose unique solution coexistence state is
$$
   (w,v)=\left( \frac{b\mu}{d\l}-1,\frac{\mu}{d}\right)\,\qquad \l> 0.
$$
As $\l\in (0,\Phi(\mu))$, $w(\l)= \frac{b\mu}{d\l}-1$ decays from $+\infty$ to $0$, while $v$ remains constant. This is the component $\mathscr{C}_0^+$ studied in Section 4 for this so special example. According to the previous analysis, for sufficiently small $\e>0$, the component $\mathscr{C}_0^+$ must perturb into a new component, $\mathscr{C}_\e^+$, having a unique coexistence state for all $\l>\Phi(\mu)$. Thus, $\mathscr{C}_\e^+$ has a supercritical turning point at some $\l_T(\e)$ such that $\lim_{\e\da 0}\l_T(\e)=0$.
\par
However, the uniqueness of the coexistence state for $\lambda>\Phi(\mu)$ can be lost when \eqref{vi.7} fails and $bc>ad$, giving rise to a $S$--shaped bifurcation diagram. Indeed, at the critical value $\l=\Phi(\mu)$, the cubic polynomial $P(w)$ becomes
\begin{equation}
\label{vi.8}
P(w)=P(w,\Phi(\mu))=w^3+\left(2-\frac{\lambda}{\varepsilon a}\right)w^2+\left(1+\frac{bc}{ad}-\frac{b\mu}{\varepsilon ad}\right)w=Q(w) w,
\end{equation}
where
$$
Q(w) := w^2+\left(2-\frac{\lambda}{\varepsilon a}\right)w+1+\frac{bc}{ad}-\frac{b\mu}{\varepsilon ad}.
$$
Thus, at $\l=\Phi(\mu)$, the roots of $P(w)$  are $w=0$ plus the two roots of $Q(w)$. A direct  calculation shows that, as soon as
$$
  \e > \frac{b\mu}{bc+ad}=\e^*,
$$
the polynomial $P(w)$ satisfies
$$
  P(0)=0,\quad P'(0) = 1+\frac{bc}{ad}-\frac{b\mu}{\varepsilon ad} >0.
$$
Suppose $\e>\e^*$ and  $Q(w)$ has two positive roots, $w_+>w_->0$. Then, at $\l=\Phi(\mu)$, the polynomial $P(w)$ has three simple roots. Thus, since the coefficients of $P(w,\l)$ are analytic functions of the parameter $\l$, for sufficiently small $\eta>0$, there are
three analytic functions
$$
  z, w_+, w_- : J_\eta\equiv (\Phi(\mu)-\eta,\Phi(\mu)+\eta)\to \R,
$$
such that
\begin{equation}
\label{vi.9}
   \lim_{\l\to \Phi(\mu)}z(\l)=0,\qquad \lim_{\l\to \Phi(\mu)}w_{\pm}(\l)=w_\pm
\end{equation}
and, for every $\l\in J_\eta$, $z(\l)$ and $w_\pm(\l)$ provide us with the three simple roots of
$P(w)$. Consequently, since $P(0)=0$, $P'(0)>0$ at $\l=\Phi(\mu)$, $P(0)<0$ if $\l>\Phi(\mu)$, and $P(0)>0$ if $\l<\Phi(\mu)$, it becomes apparent that, for sufficiently small $\eta>0$,
$$
  0 < z(\l)<w_-(\l)<w_+(\l) \quad \hbox{if}\;\; \l\in (\Phi(\mu),\Phi(\mu)+\eta),
$$
while
$$
  z(\l)<0< w_-(\l)<w_+(\l) \quad \hbox{if}\;\; \l\in (\Phi(\mu)-\eta,\Phi(\mu)).
$$
Therefore, $P(w,\l)$ has three simple positive roots if $\l\in(\Phi(\mu),\Phi(\mu)+\eta)$ and two if $\l\in(\Phi(\mu)-\eta,\Phi(\mu))$, as illustrated in
the first picture of Figure \ref{Fig8}, where we are plotting the polynomials $P(w,\l)$ for $\l=\Phi(\mu)$
(using a dashed line) and $\l_\pm=\Phi(\mu)\pm\d_\pm$ for some $\d_+,\d_-\in (0,\eta)$ (using continuous lines).
\par
Obviously, the roots of $Q(w)$ are
$$
  w_\pm  := \frac{\lambda}{2\e a}-1\pm \sqrt{\left(\frac{\lambda}{2\e a}-1\right)^2-1-\frac{bc}{ad}+
  \frac{b\mu}{\e ad} } = \frac{\lambda}{2\e a}-1\pm \frac{1}{\e a}\sqrt{\l^2-4\e^2\frac{abc}{d}}.
$$
Thus, if we further impose that
$$
  \frac{b\mu}{bc+ad} = \e^*<\e<\frac{\l}{2a},
$$
with $\e$ sufficiently close to $\e^*$, then $w_+>w_->0$ and, hence,  $P(w,\l)$ has three simple positive roots if $\l\in(\Phi(\mu),\Phi(\mu)+\eta)$ and two if $\l\in(\Phi(\mu)-\eta,\Phi(\mu))$, provided  $bc>ad$ and $\e>\e^*$ is sufficiently close to $\e^*$. The assumption $bc>ad$ is necessary and sufficient so that
$ \frac{b\mu}{bc+ad}<\frac{\l}{2a}$. Finally, since
$$
  \l-\e a w_+= \frac{\l}{2}+\e a-\frac{1}{2}\sqrt{\l^2-4\e^2\frac{abc}{d}}>\e a>0\qquad\hbox{if}\;\;\l=\Phi(\mu),
$$
by \eqref{vi.9} and \eqref{vi.3}, it becomes apparent that, if $bc>ad$ and $\e>\e^*$ is sufficiently close to $\e^*$, then \eqref{6.1} has three coexistence states
if $\l\in(\Phi(\mu),\Phi(\mu)+\eta)$ and only two if $\l\in(\Phi(\mu)-\eta,\Phi(\mu))$. This phenomenology has been illustrated in Figure \ref{Fig8}, whose right picture shows a paradigmatic $S$-shaped component $\mathscr{C}_\e^+$ for $\e>\e^*$, $\e \sim \e^*$, when $bc>ad$.

\begin{figure}[ht!]
\centering
\includegraphics[scale=0.215]{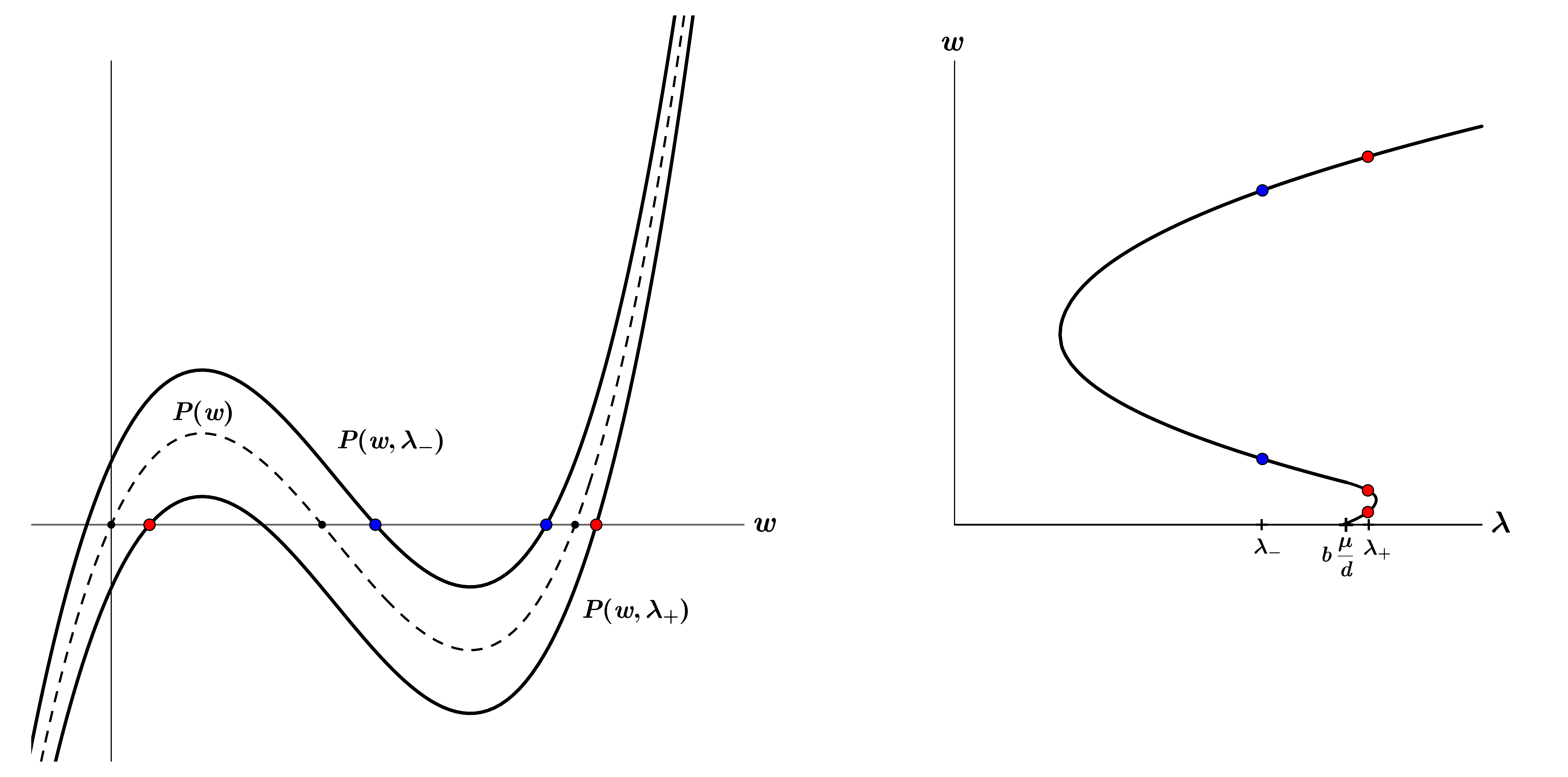}
\caption{The $S$-shaped component of constant coexistence states}
\label{Fig8}
\end{figure}

According to \eqref{vi.4}, the coefficients of $P(w,\l)$ are decreasing with respect to $\l$. Thus, in the region  $w\geq 0$, the bigger is $\l>\Phi(\mu)$, the smaller are the graphs of the polynomials $P(w,\l)$ (see the first picture of Figure \ref{Fig8}). Therefore, there exists $\l^*>\Phi(\mu)$ such that
$z(\l^*)=w_-(\l^*)$, which corresponds with the subcritical turning point of the $S$-shaped component $\mathscr{C}_\e^+$.
	
\bibliographystyle{siam}
\bibliography{LG-MH_diff}

\begin{thebibliography}{10}

\bibitem{Am-1976}
{\sc H.~Amann}, {\em Fixed point equations and nonlinear eigenvalue problems in
  ordered {Banach} spaces}, SIAM Rev., 18 (1976), pp.~620--709.

\bibitem{ALG-98}
{\sc H.~Amann and J.~L{\'o}pez-G{\'o}mez}, {\em A priori bounds and multiple
  solutions for superlinear indefinite elliptic problems}, J. Differ.
  Equations, 146 (1998), pp.~336--374.

\bibitem{BNV}
{\sc H.~Berestycki, L.~Nirenberg, and S.~R.~S. Varadhan}, {\em The principal
  eigenvalue and maximum principle for second-order elliptic operators in
  general domains}, Comm. Pure Appl. Math., 47 (1994), pp.~47--92.

\bibitem{CCLG}
{\sc S.~Cano-Casanova and J.~L\'{o}pez-G\'{o}mez}, {\em Properties of the
  principal eigenvalues of a general class of non-classical mixed boundary
  value problems}, J. Differential Equations, 178 (2002), pp.~123--211.

\bibitem{CELG}
{\sc A.~Casal, J.~C. Eilbeck, and J.~L\'{o}pez-G\'{o}mez}, {\em Existence and
  uniqueness of coexistence states for a predator-prey model with diffusion},
  Differential Integral Equations, 7 (1994), pp.~411--439.

\bibitem{Rab-71b}
{\sc M.~G. Crandall and P.~H. Rabinowitz}, {\em Bifurcation from simple
  eigenvalues}, J. Funct. Anal., 8 (1971), pp.~321--340.

\bibitem{CR73}
\leavevmode\vrule height 2pt depth -1.6pt width 23pt, {\em Bifurcation,
  perturbation of simple eigenvalues and linearized stability}, Arch. Rational
  Mech. Anal., 52 (1973), pp.~161--180.

\bibitem{Da-1983}
{\sc E.~N. Dancer}, {\em On the indices of fixed points of mappings in cones
  and applications}, J. Math. Anal. Appl., 91 (1983), pp.~131--151.

\bibitem{DLG}
{\sc D.~Daners and J.~L\'{o}pez-G\'{o}mez}, {\em Global dynamics of generalized
  logistic equations}, Adv. Nonlinear Stud., 18 (2018), pp.~217--236.

\bibitem{DL-1997}
{\sc Y.~Du and Y.~Lou}, {\em Some uniqueness and exact multiplicity results for
  a predator-prey model}, Trans. Am. Math. Soc., 349 (1997), pp.~2443--2475.

\bibitem{DL-1998}
\leavevmode\vrule height 2pt depth -1.6pt width 23pt, {\em {{\(S\)}}-shaped
  global bifurcation curve and {Hopf} bifurcation of positive solutions to a
  predator-prey model}, J. Differ. Equations, 144 (1998), pp.~390--440.

\bibitem{DL-2001}
\leavevmode\vrule height 2pt depth -1.6pt width 23pt, {\em Qualitative
  behaviour of positive solutions of a predator-prey model: {Effects} of
  saturation}, Proc. R. Soc. Edinb., Sect. A, Math., 131 (2001), pp.~321--349.

\bibitem{DS-2006}
{\sc Y.~Du and J.~Shi}, {\em A diffusive predator-prey model with a protection
  zone}, J. Differ. Equations, 229 (2006), pp.~63--91.

\bibitem{FRLG}
{\sc S.~Fern\'{a}ndez-Rinc\'{o}n and J.~L\'{o}pez-G\'{o}mez}, {\em The singular
  perturbation problem for a class of generalized logistic equations under
  non-classical mixed boundary conditions}, Adv. Nonlinear Stud., 19 (2019),
  pp.~1--27.

\bibitem{FRLG-2021}
{\sc S.~Fern{\'a}ndez-Rinc{\'o}n and J.~L{\'o}pez-G{\'o}mez}, {\em The {Picone}
  identity: a device to get optimal uniqueness results and global dynamics in
  population dynamics}, Nonlinear Anal., Real World Appl., 60 (2021), p.~41.

\bibitem{FKLGM}
{\sc J.~M. Fraile, P.~Koch~Medina, J.~L\'{o}pez-G\'{o}mez, and S.~Merino}, {\em
  Elliptic eigenvalue problems and unbounded continua of positive solutions of
  a semilinear elliptic equation}, J. Differential Equations, 127 (1996),
  pp.~295--319.

\bibitem{HFR}
{\sc H.~Freedman}, {\em Deterministic Mathematical Models in Population
  Biology}, Marcel and Dekker, New York, 1980.

\bibitem{Hen}
{\sc D.~Henry}, {\em Geometric {T}heory of {P}arabolic {D}ifferential
  {E}quations}, vol.~840 of Lectures {N}otes in {M}athematics, Springer,
  Berlin, 1981.

\bibitem{Hsu}
{\sc S.~B. Hsu}, {\em On global stability of a predator-prey system}, Math.
  Biosc., 39 (1978), pp.~1--10.

\bibitem{LG-1992}
{\sc J.~L{\'o}pez-G{\'o}mez}, {\em Positive periodic solutions of
  {Lotka}-{Volterra} reaction-diffusion systems}, Differ. Integral Equ., 5
  (1992), pp.~55--72.

\bibitem{LG01}
\leavevmode\vrule height 2pt depth -1.6pt width 23pt, {\em Spectral theory and
  nonlinear functional analysis}, vol.~426 of Chapman Hall/CRC Res. Notes
  Math., Boca Raton, FL: Chapman \& Hall/CRC, 2001.

\bibitem{LG13}
{\sc J.~L\'{o}pez-G\'{o}mez}, {\em Linear second order elliptic operators},
  World Scientific Publishing Co. Pte. Ltd., Hackensack, NJ, 2013.

\bibitem{LGMM}
{\sc J.~L\'{o}pez-G\'{o}mez and M.~Molina-Meyer}, {\em The maximum principle
  for cooperative weakly coupled elliptic systems and some applications},
  Differential Integral Equations, 7 (1994), pp.~383--398.

\bibitem{LGMH20}
{\sc J.~L\'{o}pez-G\'{o}mez and E.~Mu{\~{n}}oz-Hern\'{a}ndez}, {\em A spatially
  heterogeneous predator-prey model}, Discrete Contin. Dyn. Syst. Ser. B, 26
  (2021), pp.~2085--2113.

\bibitem{LGP-1992}
{\sc J.~L{\'o}pez-G{\'o}mez and R.~Pardo San~Gil}, {\em Coexistence regions in
  {Lotka}-{Volterra} models with diffusion}, Nonlinear Anal., Theory Methods
  Appl., 19 (1992), pp.~11--28.

\bibitem{LGRab}
{\sc J.~L{\'o}pez-G{\'o}mez and P.~H. Rabinowitz}, {\em The effects of spatial
  heterogeneities on some multiplicity results}, Disc. Cont. Dyn. Systems, 36
  (2016), pp.~941--952.

\bibitem{May}
{\sc R.~May}, {\em Stability and Complexity in Model Ecosystems}, Princeton
  University Press, Princeton, 1974.

\bibitem{Why-1964}
{\sc G.~T. Whyburn}, {\em Topological analysis}, vol.~23 of Princeton Math.
  Ser., Princeton University Press, Princeton, NJ, 1964.

\end{thebibliography}

\end{document}